\documentclass[12pt,amscd]{amsart}
\newtheorem{thm}{Theorem}[section]
\usepackage[all]{xy}
\usepackage{graphicx}
\usepackage{amsmath,amsxtra,amssymb,latexsym, amscd,amsthm}
\usepackage{indentfirst}
\usepackage[mathscr]{eucal}
\usepackage{hyperref}
\hypersetup{colorlinks=true,linkcolor=red,filecolor=magenta,urlcolor=green}

\newtheorem{cor}[thm]{Corollary}
\newtheorem{lem}[thm]{Lemma}
\newtheorem{prop}[thm]{Proposition}
\theoremstyle{definition}
\newtheorem{defn}[thm]{Definition}
\newtheorem{exam}[thm]{Example} 
\newtheorem{rem}[thm]{Remark}

\DeclareMathOperator{\N}{\mathbb {N}}
\DeclareMathOperator{\Z}{\mathbb {Z}}

\DeclareMathOperator{\depth}{depth}
\DeclareMathOperator{\ch}{char}

\DeclareMathOperator{\lk}{lk}
\DeclareMathOperator{\reg}{reg}
\DeclareMathOperator{\pd}{pd}
\DeclareMathOperator{\dst}{dstab}
\DeclareMathOperator{\sdst}{sdstab}
\DeclareMathOperator{\length}{length}

\def\alb {\boldsymbol {\alpha}}
\def\btb {\boldsymbol {\beta}}

\def\zv {\mathbf 0}

\def\x {\mathbf x}

\def\mi {\mathfrak m}

\def\h {\widetilde{H}}

\def\F {\mathcal{F}}

\begin{document}

\title[Stability index of depth function]{Depth stability of cover ideals}

\author[M.P. Binh]{Mai Phuoc Binh}
\address{University of Transport and Communications , No.3 Cau Giay Street, Lang Thuong ward, Dong Da District, Hanoi, Vietnam}
\email{binhmp@utc.edu.vn}

\author[N.T. Hang]{Nguyen Thu Hang}
\address{International Centre of Research and Postgraduate Training in Mathematics, 18B Hoang Quoc Viet Street, Ha Noi, Vietnam}
\address{Thai Nguyen University of Sciences, Tan Thinh Ward, Thai Nguyen City, Thai Nguyen, Vietnam}
\email{hangnt@tnus.edu.vn}

\author[T.T. Hien]{Truong Thi Hien}
\address{Hong Duc University, 565 Quang Trung Street, Dong Ve Ward, Thanh Hoa, Viet Nam}
\email{hientruong86@gmail.com}

\author[T.N. Trung]{ Tran Nam Trung}
\address{Institute of Mathematics, VAST, 18 Hoang Quoc Viet, Hanoi, Viet Nam, and Institute of Mathematics and TIMAS, Thang Long University, Ha Noi, Vietnam.}
\email{tntrung@math.ac.vn}

\subjclass{13A15, 13C15, 05C90, 13D45.}
\keywords{Ordered matching, Cover ideal, Depth function, Symbolic power}
%\date{}

%\commby{}
%-----------------------------------------------------------
\begin{abstract} Let $R=K[x_1,\ldots,x_r]$ be a polynomial ring over a field $K$. Let $G$ be a graph with vertex set $\{1,\ldots,r\}$ and let $J$ be the cover ideal of $G$. We give a sharp bound for the stability index of symbolic depth function $\sdst(J)$. In the case $G$ is bipartite, it yields a sharp bound for the stability index of depth function $\dst(J)$ and this bound is exact if $G$ is a forest.
\end{abstract}

% -----------------------------------------------------------
\maketitle
% -----------------------------------------------------------
\section*{Introduction}
Let $R = K[x_1,\ldots,x_r]$ be a polynomial ring over a field $K$ and let $I$ be a homogeneous ideal of  $R$. We call the functions $\depth R/I^n$ and $\depth R/I^{(n)}$, for $n\geqslant 1$, the {\it depth function} and the {\it symbolic depth function} of  $I$, respectively.

It is a classical result of Brodmann \cite{B} that the depth function of an ideal in a Noetherian ring is asymptotically a constant function. The same result is not true for the symbolic depth function (see \cite[Theorem 4.4]{NT}), but it is true for square-free monomial ideals (see \cite{HKTT, HT2}).

By virtue of these results we define the {\it stability index of depth function} of $I$ to be
$$\dst(I) = \min\{n_0\mid \depth R/I^n =\lim\limits_{k\to \infty} \depth R/I^k, \text{ for all } n\geqslant n_0\}$$ 
and the {\it stability index of symbolic depth function} of $I$ to be
$$\sdst(I) = \min\{n_0\mid \depth R/I^{(n)} =\lim\limits_{k\to \infty} \depth R/I^{(k)}, \text{ for all } n\geqslant n_0\}$$
provided the second limit exists. 

It is of great interest to bound effectively $\dst(I)$ and $\sdst(I)$ for square-free monomial ideals. However these problems seem to be very difficult and are currently solved for only few cases (see e.g. \cite{HHTT,HT,  HHbi, HHT2, HRV, HV, HKTT, MT2, Tr2}).

In our paper,  we investigate these indices on the class of square-free monomial ideals of height $2$, these ideals can  be defined as cover ideals of graphs.  Let $G=(V,E)$ be a graph with vertex set $V=\{1,\ldots,r\}$. For a subset $\tau$ of $V$, denote the square-free monomial $\x_{\tau}$ to be the product of all variables $x_i$ where $i\in\tau$. Then, the {\it cover ideal} of $G$ is defined by:
$$J(G) = (\x_{\tau} \mid \tau \text{ is a minimal vertex cover of } G).$$

In order to find a good bounds for $\dst(J(G))$ and $\sdst(J(G))$ we use the theory of matchings in graphs. We mainly deal with {\it ordered matchings} which are defined as follows.  A matching $M=\{\{u_i,v_i\} \mid i=1,\ldots,s\}$ in a graph $G$ is called an ordered matching if:
\begin{enumerate}
\item $\{u_1,\ldots,u_s\}$ is an independent set in $G$,
\item $\{u_i, v_j\} \in E$ implies $i \leqslant j$.
\end{enumerate}
In this case, the set $A=\{u_1,\ldots,u_s\}$ is called	 the {\it free parameter set } of  $G$ and $B=\{v_1,\ldots,v_s\}$ is called the {\it partner set} of $A$.

\medskip

For an ordered matching $M$ in a graph $G$, an $M$-{\it alternating path} is a path whose edges are alternately in $M$ and $E\setminus M$. This notion plays an important role in the theory of matchings (see e.g. \cite{BM}). For an ordered matching $M$ in $G$, let $\ell(M)$ be the length of a longest $M$-alternating path, and
$$\ell(G) = \min\{\ell(M) \mid M \text{ is a maximum ordered matching in } G\}.$$
Although an $M$-alternating path may be not simple but these numbers are always finite (see Lemmas  \ref{lem-ell} and  \ref{up-nu}).

\medskip

Then the  main result of our paper is the following theorem.

\medskip

\noindent{\bf Theorem 1.} (see Theorem \ref{non-bipartite-theorem}) 
{$\sdst(J(G)) \leqslant (\ell(G)+1)/2$ for any graph $G$.}

\medskip

Note that the bound in  Theorem 1 improves the bound given in \cite[Theorem 3.6]{HT} for all graphs, and  the one in \cite[Theorem 3.4]{HKTT} for bipartite graphs. 

\medskip

Moreover, the equality holds for a broad class of graphs. Let $\nu(G)$ and $\nu_0(G)$ denote the matching and the ordered matching numbers of $G$, respectively. Then,  we have the following theorem  (see Propositions \ref{PM}, \ref{EPM} and \ref{main-cor}).

\medskip

\noindent{\bf Theorem 2.} {$\sdst(J(G)) = (\ell(G)+1)/2$ if $G$ is one of the following graphs:
\begin{enumerate}
\item $G$ has a perfect ordered matching, i.e. $|V(G)| = 2\nu_0(G)$, or
\item $\nu(G) = \nu_0(G)$ and $G$ contains no cycles of length $5$, or
\item $G$ is a forest.
\end{enumerate}
}

\medskip

It is worth mentioning that $\dst(J(G))$ and $\sdst(J(G))$ depend on the characteristic of the base filed $K$ (see Example \ref{char-depend}). This means that we can not give a combinatorial formula for these indices as in Theorem 2.

%%%%%%%%%%%%%%%%%%%%%%%%%%%%%%%%%%%%%%%%%%%%%%%%%%%%%%%%%%
\section{Preliminary}

In this section, we recollect notation, terminology and basic results used in the paper. We follow standard texts \cite{BM, BH, MS}. Throughout the paper, let $K$ be a field, let $R = K[x_1,\ldots,x_r],\  r\geqslant 1$ be a polynomial ring, and let $\mi = (x_1,\ldots,x_r)$ be the maximal homogeneous ideal of $R$.

\subsection{Depth and regularity} The object of our work is the depth of graded modules and ideals over $R$. This invariant can be defined via either the minimal free resolutions or the local cohomology modules. 

Let $M$ be a nonzero finitely generated graded  $R$-module and let
$$0 \rightarrow \bigoplus_{j\in\Z} R(-j)^{\beta_{p,j}(M)} \rightarrow \cdots \rightarrow \bigoplus_{j\in\Z}R(-j)^{\beta_{0,j}(M)}\rightarrow 0$$
be the minimal free resolution of $M$. The {\it projective dimension} of $M$ is the length of this resolution
$$\pd(M) = p,$$
and the {\it depth} of $M$ is given by Auslander-Buchsbaum formula
$$\depth(M) =r- p.$$ 
Another invariant measures the complexity of the resolution is the \emph{Castelnuovo–Mumford regularity} (or regularity for short) of $M$ which is defined by
$$\reg(M) = \max\{j-i\mid \beta_{i,j}(M)\ne 0\}.$$

The depth and regularity of $M$ can also be computed via the local cohomology modules of $M$. Let $H_{\mi}^i(M)$ be the $i$-th cohomology module of $M$ with support in $\mi$. Then,
$$\depth(M) = \min\{i\mid H_{\mi}^i(M) \ne 0\},$$
and
$$\reg(M) = \max\{j + i\mid H_{\mi}^i(M)_j \ne \zv, \text{ for } i = 0,\ldots, \dim(M), \text{ and } j\in \Z\}.$$

\subsection{Graphs}

Let $G$ be a simple graph. We use the symbols $V(G)$ and $E(G)$ to denote the vertex set and the edge set of $G$, respectively. In this paper we always assume that $E(G)\ne \emptyset$ unless otherwise indicated.

 An edge $e$ is incident to a vertex $v$ if $v\in e$. The {\it degree} of a vertex is the number of edges incident to $u$, $u$ is a leaf if it has degree one. If  $e =\{u,v\}$, then  $u$ and $v$ are  ends of $e$ or endpoints of $e$, $u$ and $v$ are adjacent to  each other. If there is no confusion, we simply write $e = uv$. 

A graph $H$ is called a {\it subgraph} of $G$ if $V(H)\subseteq V(G)$ and $E(H) \subseteq E(G)$.  A graph $H$ is called an {\it induced subgraph} of $G$ if the vertices of $H$ are vertices in $G$, and for vertices $u$ and $v$ in $V(H)$, $\{u,v\}$ is an edge of $H$ if and only if $\{u,v\}$ is an edge in $G$. The induced subgraph of $G$ on a subset $S \subseteq V(G)$, denoted by $G[S]$, is obtained by deleting vertices not in $S$ from $G$ (and their incident edges).

For a subset $M\subseteq E(G)$, denote $V(M)$ to be all vertices of $G$ that are endpoins of edges in $M$ and denote $G[M]$ to be $G[V(M)]$.
\medskip

Let $p\colon v_0,v_1,\ldots, v_k$ be a sequence of vertices of $G$. Then,
\begin{enumerate}
\item $p$ is called a {\it path} if $\{v_{i-1},v_i\}\in E(G)$ for $i=1, \ldots,k$. In this case, we say that $p$ is a path from $v_0$ to $v_k$.
\item $p$ is called a {\it simple path} if it is a path and every vertex appears exactly once.
\item $p$ is called a {\it cycle} if $k\geqslant 3$ and $p$ is a path with distinct vertices except for $v_0 = v_k$.
\end{enumerate}
In each case, $k$ is called the {\it length} of $p$. A simple path is longest if it is among the simple paths of largest length of $G$.

If $p\colon v_0,v_1,\ldots, v_k$ is the path from $v_0$ to $v_k$ we indicate it by $v_0 \to v_1 \to \cdots \to v_k$, and $v_0$ and $v_k$ are called the origin and the terminus of $p$, respectively.

A graph $G$ with $r$ vertices such that all edges lying on a simple path is called a path with $r$ vertices, denoted by $P_r$. Similarly, if all edges of $G$ lying on a cycle, then $G$ is called a cycle of $r$ vertices, denoted by $C_r$. As usual, the cycle $C_3$ is call a {\it triangle}, the cycle $C_4$ is called a {\it quadrilateral}, the cycle $C_5$ is called a {\it pentagon}, and so forth.

A  graph is {\it connected} if there is a path from any point to any other point in the graph. A graph that is not connected is said to be {\it disconnected}. A connected component of  a graph $G$ is a connected subgraph that is not part of any larger connected subgraph. The components of any graph partition its vertices into disjoint sets, and are the induced subgraphs of those sets.

The graph $G$ is {\it bipartite} if  $V(G)$ can be partitioned into two subsets $X$ and $Y$ such that every edge has one end in $X$ and another end in $Y$; such a partition $(X,Y)$ is called a {\it bipartition} of the graph. Note that $G$ is bipartite if and only if it has no cycle of odd length (see \cite[Theorem 4.7]{BM}).  A connected graph without cycles is a {\it tree}. Obviously, a tree is bipartite. A graph is a {\it forest} if  every its connected component  is a tree.

A {\it matching} in the graph $G$ is a set of pairwise non adjacent edges. If $M$ is a matching, the two ends of each edge of $M$ are said to be matched under $M$, and each vertex incident with an edge of $M$ is said to be covered by $M$. A {\it perfect matching} is one which covers every vertex of the graph, a maximum matching one which covers as many vertices as possible. The number of edges in a maximum matching in a graph $G$ is called the {\it matching number} of $G$ and denoted $\nu(G)$.

A matching $M$ of $G$ is called an {\it induced matching} if the graph $G[M]$ is just disjoint edges. The {\it induced matching number} of $G$, denoted by $\nu'(G)$, is the maximum size of an induced matching in $G$.

An {\it independent set} in $G$ is a set of vertices no two of which are adjacent to each other.  According to Constantinescu and Varbaro \cite{CV}, we define an ordered matching as follows.

\begin{defn}\label{ordered-matching} A matching $M=\{\{u_i,v_i\} \mid i=1,\ldots,s\}$ in a graph $G$ is called an {\it ordered matching} if:
\begin{enumerate}
\item $\{u_1,\ldots,u_s\}$ is an independent set in $G$,
\item $\{u_i, v_j\} \in E(G)$ implies $i \leqslant j$.
\end{enumerate}
In this case, the set $A=\{u_1,\ldots,u_s\}$ is called	 the {\it free parameter set } of  $G$ and $B=\{v_1,\ldots,v_s\}$ is called the {\it partner set} of $A$. For simplicity, we also say $A$ and $B$ are free parameter set and partner set for $M$.  

The {\it ordered matching number} of $G$, denoted by $\nu_0(G)$ is the maximum size of an ordered matching in $G$. 
\end{defn}

An ordered matching in $G$ is a perfect matching is called a {\it perfect ordered matching}. Thus, $G$ has a perfect order matching if  and only if $|V(G)| = 2\nu_0(G)$. 

For example, the graph $G$ depicted in Figure $1$ has a perfect ordered matching 
$$M =\{ \{1,5\}, \{2,6\}, \{3,7\}, \{4,8\}\}$$ 
which are bold edges in the figure. In this case, $A=\{1,2,3,4\}$ and $B = \{5,6,7,8\}$.

\begin{center}

\includegraphics[scale=0.5]{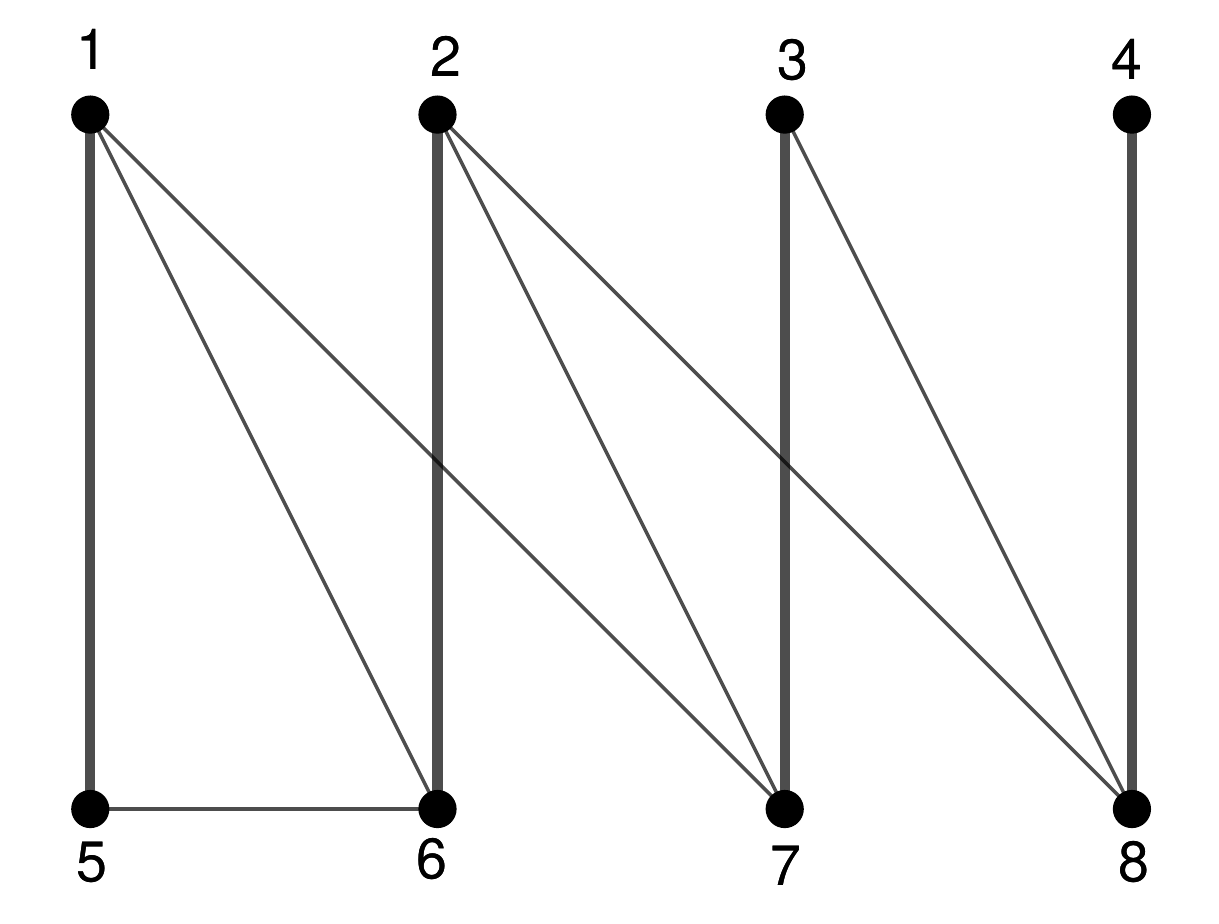}\\
\medskip

{\it Figure $1$. A graph with a perfect ordered matching}
\end{center}

\begin{lem}\label{unique-matching} If a graph $G$ has a perfect ordered matching, then it has a unique perfect matching.
\end{lem}
\begin{proof} Let $M = \{a_ib_i \mid i=1,\ldots,s\}$ be a perfect ordered matching in $G$. Let $M'$ be a perfect matching in $G$. We will prove by induction on $s$  that $M' = M$ as subsets of $E(G)$. Indeed, if $s=1$, then $G$ is just the edge $a_1b_1$, and the lemma is obvious true.

Assume that $s \geqslant 2$. Since $M$ is a perfect ordered matching, we have $a_s$ is a leaf. Recall that $M'$ is a perfect matching, so $a_sb_s\in M'$.

Since $M\setminus \{a_sb_s\}$ and $M'\setminus \{a_sb_s\}$ are a perfect ordered matching and a perfect matching in $G\setminus \{a_s,b_s\}$, respectively, by the induction hypothesis we have $M\setminus \{a_sb_s\} = M'\setminus \{a_sb_s\}$. Thus, $M = M'$, as required.
\end{proof}

Let $M$ be a matching of $G$. An $M$-{\it alternating path} is a path whose edges are alternately in $M$ and $E\setminus M$. This notion plays an important role in the theory of matchings (see \cite{BM}). 

\begin{defn}\label{ord-len} Let $M$ be an ordered matching in $G$.  Let
$$\ell(M) = \max\{\length(p)\mid p \text{ is an $M$-alternating path}\},$$
and 
$$\ell(G) =\min\{\ell(M) \mid M \text{ is a maximum ordered matching of } G \}.$$
\end{defn}

\begin{exam} Let $G$ be a graph depicted in Figure $1$.  Then, $G$ has a perfect ordered matching $M=\{\{1,5\},\{2,6\}, \{3,7\},\{4,8\}\}$. We can see that 
$$4 \to 8 \to 3 \to 7 \to 2 \to 6 \to 1 \to 5 \to 6\to 2 \to 7\to 3 \to 8 \to 4,$$
is an $M$-alternating path. Note that it is not simple. 

From the Figure 1, we can see that it is a longest $M$-alternating path, so $\ell(M)=13$. Now we compute $\ell(G)$. Let $M'$ be another perfect ordered matching of $G$. Then, $M = M'$ as subsets of $E(G)$ by Lemma \ref{unique-matching}, and then $\ell(M') =\ell(M)$. Therefore, $\ell(G) =\ell(M) = 13$.
\end{exam}

The argument in the last paragraph in  this example  yields the following useful fact.

\begin{rem}\label{perfect-ell} If a graph $G$ has a perfect ordered matching, then $\ell(G) = \ell(M)$ for every perfect ordered matching $M$.
\end{rem}

In this paper we introduce the following type of alternating paths.

\begin{defn} \label{admissible-path} Let $M$ be an ordered matching in $G$ with free parameter and partner sets $A$ and $B$, respectively. For a vertex $v$ covered by $M$, we called an $M$-alternating path $p$ is an {\it $M$-admissible path for $v$} if
\begin{enumerate}
\item $v$ is the origin of $p$,
\item the first and the last edges are in $M$,
\item every edge of $p$ has an end in $A$ and another end in $B$. 
\end{enumerate}
\end{defn}

We next show that any path in the definition above  is simple.

\begin{lem} \label{simple-path}Let $M$ be an ordered matching in a graph $G$ with free parameter set $A$ and partner set $B$.  Assume that $B$ is an independent set of $G$. Then, every $M$-alternating path $p$ with the first edge and the last edge in $M$ is simple.
\end{lem}
\begin{proof}  By the assumption, the length of $p$ is odd. Assume the path $p$ is of the form:
$$u_1\to v_1\to u_2\to v_2\cdots \to u_{k-1}\to v_{k-1}\to u_k \to v_k,$$
so that $\{u_1v_1, \ldots, u_kv_k \} \subseteq M$. Let $M = \{e_1,\ldots,e_s\}$. Then, for $i=1,\ldots, k$, there is $\sigma(i) \in \{1,\ldots, s\}$ such that $u_iv_i = e_{\sigma(i)}$. Hence, in order to prove the lemma it suffices to show that the sequence $\{\sigma(i)\}_{i=1}^k$ is strictly monotone.  To prove this, we consider two cases:

{\it Case 1}:  $u_1\in B$. Since $A$ and $B$ are independent sets of $G$, we deduce that every edge has an end in $A$ and another end in $B$.  This fact implies that  $u_1,\ldots, u_k\in B$ and $v_1,\ldots,v_k \in A$. 

We now prove the sequence is increasing. Fix an index $i$, we need to show that $\sigma(i) < \sigma(i+1)$. Indeed, since $v_i\in A$,  $u_{i+1}\in B$ and $v_i u_{i+1}\in E(G)\setminus M$, by the definition of ordered matching we deduce that $\sigma(i) < \sigma(i+1)$.

{\it Case 2}: $u_1\in A$. In this case, $u_1,\ldots, u_k\in A$ and $v_1,\ldots,v_k \in B$. By the similar argument as in the previous case we can prove that the sequence is decreasing, and the lemma follows.
\end{proof}

\begin{lem} \label{simple-admissible} Let $M$ be an ordered matching in a graph $G$ and let $p$ be an $M$-admissible for a vertex $v\in V(M)$. Then,
\begin{enumerate}
\item $p$ is simple.
\item $\length(p)$ is odd.
\end{enumerate}
\end{lem}
\begin{proof} Let $A$ and $B$ be the free parameter and the partner sets for $M$, respectively. Let $G(A,B)$  be the bipartite graph where the vertex set is $A \cup B$, and the edges of $G(A,B)$ are edges of $G$ with an end in $A$ and another end in $B$. Thus, the couple $(A,B)$ is a bipartition of $G(A,B)$. 

Then, for a path $p$ we have each edge of $p$ has an end in $A$ and an end in $B$ if and only if $p$ is a path in $G(A,B)$. Thus,  (1) deduced from Lemma \ref{simple-path}.

We now prove $(2)$. Since $p$ is an $M$-alternating path with the first and the last edges in $M$. It follows that $\length(p) = 2k-1$ where $k$ is the number of edges in $M$ lying in $p$. This proves the lemma.
\end{proof}

\begin{lem}\label{lem-ell-1} Let $M$ be an ordered matching in a graph $G$ with the free parameter set $A$ and the partner set $B$. Then, every $M$-alternating path with the first and the last edges in $M$ is either an $M$-admissible path of a vertex in $B$ or a join of two $M$-admissible paths of two vertices $u$ and $v$ in $B$ via the edge $uv$. 
\end{lem}
\begin{proof} Let $p$ be an $M$-alternating path with the first and the last edges in $M$. Then, $\length(p)$ is odd so we may assume that $\length(p) = 2k-1$ and $$p\colon u_1\to v_1\to u_2\to v_2\to \cdots \to u_k\to v_k$$ where $u_iv_i \in M$ for $i=1,\ldots, k$. We consider two cases:

\medskip

{\it Case 1}: $u_1 \in B$. Since $A$ is an independent set of $G$, we deduce that
$$u_1,\ldots, u_k\in B \text{ and } v_1,\ldots,v_k \in A.$$
This means that $p$ is an $M$-admissible path for $u_1$. 

\medskip

{\it Case 2}: $u_1 \in A$. If $p$ is an $M$-admissible path for $u_1$, then $p$ has the desired form. Now assume that $p$ is not an $M$-admissible path for $u_1$. Then, $v_iu_{i+1} \in E(G[B])$ for some $i\in\{1,\ldots,k-1\}$. Let $i$ be the smallest such integers. Let $p_1$ be the first part of $p$ from $u_1$ to $v_i$, and $p_2$ the second part of $p$ from $u_{i+1}$ to the terminus of $p$. Then, $p_2$ is an $M$-admissible path for $v_{i+1}$ by the argument in Case 1 above. Note that $p_1$ is an $M$-admissible path for $v_i$ if we reverse its direction. Thus, $p$ is a join of $p_1$ and $p_2$ via the edge $v_iu_{i+1}$, and  the lemma follows.
\end{proof}

In order to compute $\ell(M)$ we introduce the following notations.

\begin{defn}\label{ell-0} Let $M$ be an ordered matching in a graph $G$ with free parameter set $A$ and partner set $B$. For every vertex $v$ covered by $M$, let $\ell(v)$ be the length of a longest $M$-admissible path of $v$. Let
$$\ell_0(M) = \max\{\ell(v)\mid v\in B\}, \ \ell_1(M) =  \max\{\ell(u)+\ell(v)+1 \mid uv\in E(G[B])\},$$
where we make a convention that $\ell_1(M) = 0$ if $B$ is an independent set of $G$.
\end{defn}

\begin{lem}\label{lem-ell} Let $M$ be an ordered matching in a graph $G$. Then, $$\ell(M) =\max\{\ell_0(M), \ell_1(M)\}.$$
\end{lem}
\begin{proof} Let $p$ be any $M$-alternating path. We first claim that 
$$\length(p) \leqslant \max\{\ell_0(M), \ell_1(M)\}.$$
Indeed, we consider four possible cases:

\medskip

{\it Case 1}: The first and the last edges of $p$ are in $M$, then the desired inequality follows from Lemma  \ref{lem-ell-1}.

{\it Case 2}:  The first edge of $p$ is not in $M$ but the last one is in $M$. Let $u$ is the origin of $p$ and let $v$ is a vertex such that $uv\in M$. Then, by adding $v$ to the beginning of $p$ we obtain a path $p'$ with the first and the last edges in $M$. By Lemma \ref{lem-ell-1} we imply that
$$\length(p) = \length(p')-1< \length(p')\leqslant \max\{\ell_0(M), \ell_1(M)\}.$$

{\it Case 3}: The first edge of $p$ is in $M$, but the last one is not in $M$. By the same argument as in Case $2$, we can add a suitable vertex at the end of $p$ to get a path $p'$ with the first and the last edges in $M$. Then,
$$\length(p) =\length(p')-1<\length(p')\leqslant \max\{\ell_0(M), \ell_1(M)\}.$$

{\it Case 4}: The first and the last edges of $p$ are not in $M$. Then, by the same argument as in Case $2$, we can add a vertex to the beginning of $p$ and a vertex at the end of $p$ to get an $M$-alternating path $p'$ with the first and the last edges in $M$. Then,
$$\length(p) =\length(p')-2< \length(p') \leqslant \max\{\ell_0(M), \ell_1(M)\},$$
and the claim follows.

\medskip

We now prove the lemma. By the claim above, $\ell(M) \leqslant \max\{\ell_0(M), \ell_1(M)\}$. The reverse inequality follows from Lemma \ref{lem-ell-1}, and the proof is complete.
\end{proof}

\begin{exam} Let $G$ be the graph depicted in Figure $2$. 

\begin{center}
\includegraphics[scale=0.5]{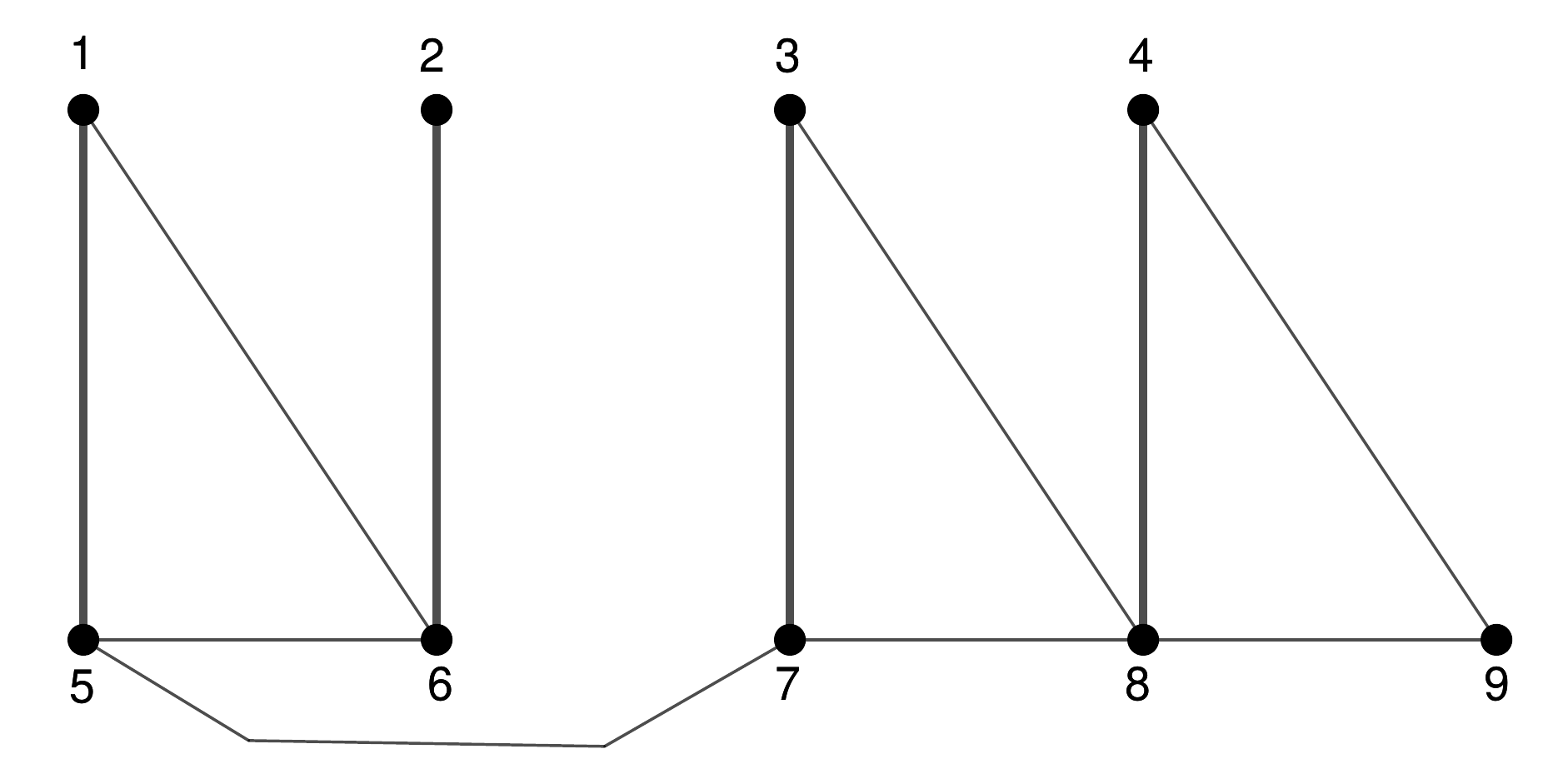}\\
\medskip

{\it Figure $2$. Graph $G$ with $\ell(M) = 7$}
\end{center}

\medskip

We have the set  $M = \{\{1,5\}, \{2,6\}, \{3,7\}, \{4,8\}\}$ is a maximum ordered-matching in $G$. In particular, $\nu_0(G) = \nu(G) = 4$, and  $M$ has the free parameter set $A =\{1,2,3,4\}$ and the partner set $B = \{5,6,7,8\}$. 

Observe that $G(A,B)$ consists of two disjoint paths: $\{2,6,1,5\}$ and $\{4,8,3,7\}$ of length $3$. Thus,
$$\ell(1) = 1, \ \ell(2)=3, \ \ell(5) = 3, \ \ell(6)=1,$$
and
$$\ell(3) = 1, \ \ell(4)=3, \ \ell(7) = 3, \ \ell(8)=1,$$
so that $\ell_0(M) = \max\{\ell(v)\mid v=5,6,7,8\} = 3$ and
$$\ell_1(M)= \max\left\{\ell(u)+\ell(v)+1 \mid uv\in E(G[B])\right\} = \ell(5)+\ell(7)+1 = 7.$$
Thus, $\ell(M) = \max\{\ell_0(M), \ell_1(M)\} = 7$. Note also that $$2\to 6 \to 1 \to 5 \to 7 \to 3 \to 8 \to 4,$$
is a longest $M$-admissible path. It is the join of two $M$-admissible paths of $5$ and $7$, respectively, via the edge $\{5,7\}$.

In order to compute $\ell(G)$, we consider the maximum ordered matching $$M'=\{\{1,5\}, \{2,6\}, \{4,9\}, \{3,8\}\}.$$
With this ordered matching, we have $\ell(M') = 4$, thus $\ell(G) \leqslant 4$. 

We will show that $\ell(G)=4$ by showing that $\ell(M'')\geqslant 4$ for every maximum ordered matching $M''$ of $G$. In order to prove this, observe that $G[M'']$ is obtained from $G$ by deleting one vertex. By consider all cases when removing one vertex from $G$ and using Remark \ref{perfect-ell} we can verify that $\ell(M'')\geqslant 4$, as required.
\end{exam}

By Lemma \ref{lem-ell} we deduce that $\ell(G)$ is finite, the next result gives explicitly bounds for this number.

\begin{lem}\label{up-nu} Let $G$ be a graph. Then,
\begin{enumerate}
\item $\ell(G) \leqslant 4\nu_0(G)-3$.
\item If $G$ is bipartite, then $\ell(G) \leqslant 2\nu_0(G)-1$.
\end{enumerate}
\end{lem}
\begin{proof} (1)  Let $M$ be a maximum ordered matching of $G$ with the free parameter and the partner sets $A$ and $B$, respectively. By Lemma \ref{lem-ell}, it suffices to show that $\ell_0(M)\leqslant 4\nu_0(G)-3$ and $\ell_1(M)\leqslant 4\nu_0(G)-3$. First, for an $M$-admissible path $p$ for a vertex $v\in B$, since $p$ is a simple path in $G[M]$, it follows that
$$\length(p) \leqslant |V(G[M])|-1 = 2\nu_0(G)-1\leqslant 4\nu_0(G)-3.$$
Hence, $\ell_0(M)\leqslant 4\nu_0(G)-3$.

Now, assume that $\ell_1(M) = \ell(u)+\ell(v)+1$ where $uv\in E(G[B])$. We may assume that $\ell(u)\leqslant \ell(v)$, hence $\ell(u)\leqslant \ell(v) \leqslant 2\nu_0(G)-1$, where the last inequality follows from the same argument as in the previous paragraph. Let $p_1$ and $p_2$ be two $M$-admissible paths for $u$ and $v$, respectively, with $\length(p_1)=\ell(u)$ and $\length(p_2) = \ell(v)$.  If $\ell(u) = 2\nu_0(G)-1$, then $p_1$ is a longest $M$-admissible path in $G(A,B)$. It follows that $u$ is a leaf in $G(A,B)$. In particular, $u$ is not in $p_2$ since $p_2$ has the terminus in $A$. But then, $\length(p_2) \leqslant 2\nu_0(G)-2 < \length(p_1)$, a contradiction. Thus, $\length(p_1) < 2\nu_0(G)-1$. Since $\length(p_1)$ is odd, we have $\length(p_1)\leqslant 2\nu_0(G)-3$. Hence,
$$\ell_1(M) = \length(p_1)+\length(p_2)+1\leqslant 2\nu_0(G)-3+2\nu_0(G)-1 + 1 = 4\nu_0(G)-3,$$
and  Part 1 follows.
\medskip

(2)  Assume that $G$ is bipartite so that $G[M]$ is  bipartite. Let $(X,Y)$ be a bipartition of $G[M]$. Note that $M$ is a perfect ordered matching in $G[M]$. By \cite[Lemma 3.4]{HT}, $G[M]$ has a perfect ordered matching $M'$ with the free parameter set $X$ and the partner set $Y$. Since $M$ and $M'$ are perfect matchings of $G$, so that $M' = M$ as subsets of $E(G)$ by Lemma \ref{unique-matching}. In particular, $\ell(M) = \ell(M')$. Now, we have $\ell_1(M')=0$, so that $\ell(M') = \ell_0(M')$. The inequality $\ell_0(M') \leqslant 2\nu_0(G)-1$ is obvious, so $\ell(G) \leqslant 2\nu_0(G)-1$, and the proof is complete.
\end{proof}

We conclude this section with the following remark about $M$-alternating paths that are taken from the proofs of Lemmas \ref{lem-ell} and \ref{up-nu}.

\begin{rem} Let $G$ be a graph and $M$ be ordered matching in $M$. Then,
\begin{enumerate}
\item Every longest $M$-alternating path has the first and the last edges in $M$.
\item If $G$ is bipartite, then every $M$-alternating path is simple.
\end{enumerate}
\end{rem}

\subsection{Simplicial complexes and Stanley-Reisner ideals} We recall a relationship between cover ideals of graphs and simplicial complexes. A {\it simplicial complex} on $V = \{1,\ldots, r\}$ is a collection of subsets of $V$ such that if $\sigma\in \Delta $ and $\tau\subseteq \sigma$ then $\tau\in \Delta$. Elements of $\Delta$ are called faces. Maximal faces (with respect to inclusion) are called facets. For $F \in \Delta$, the dimension of $F$ is defined to be $\dim F = |F|-1$. The empty set, $\emptyset$, is the unique face of dimension $-1$, as long as $\Delta$ is not the void complex $\{\}$ consisting of no subsets of $V$.  The link of $F$ inside $\Delta$ is its subcomplex:
$$\lk_{\Delta}(F) = \{H\in \Delta \mid H\cup F\in \Delta \ \text{ and } H\cap F=\emptyset\}.$$

Every element in a face of $\Delta$ is called a {\it vertex} of $\Delta$. Let us denote $V(\Delta)$ to be the set of vertices of $\Delta$. If there is a vertex, say $j$, such that $\{j\}\cup F \in \Delta$ for every $F\in \Delta$, then $\Delta$ is called a {\it cone} over $j$. It is well-known that if $\Delta$ is a cone, then it is an acyclic complex. Recall that a chain complex is called an {\it acyclic} complex if all of whose homology groups are zero. 

For a subset $\tau = \{j_1,\ldots,j_i\}$ of $[r]$, denote $\x_\tau = x_{j_1} \cdots x_{j_i}$. Let $\Delta$ be a simplicial complex over the set $V=\{1,\ldots,r\}$. The Stanley-Reisner ideal of $\Delta$ is defined to be the squarefree monomial ideal
$$I_{\Delta} = (\x_\tau \mid \tau \subseteq [r] \text{ and } \tau \notin \Delta) \ \text{ in } R = K[x_1,\ldots,x_r]$$
and the {\it Stanley-Reisner} ring of $\Delta$ to be the quotient ring $k[\Delta] = R/I_{\Delta}$.  This provides a bridge between combinatorics and commutative algebra (see \cite{MS,ST}). 

Note that if $I$ is a square-free monomial ideal, then it is a Stanley-Reisner ideal of the simplicial complex $\Delta(I)= \{\tau \subseteq [r] \mid \x^\tau\not \in I\}$. When $I$ is a monomial ideal (maybe not square-free) we also use $\Delta(I)$ to denote the simplicial complex corresponding to the square-free monomial ideal $\sqrt{I}$. 

Let $\mathcal{F}(\Delta)$ be the set of facets of $\Delta$. If $\mathcal{F}(\Delta) =\{F_1,\ldots,F_m\}$, we write $\Delta = \left<F_1,\ldots,F_m\right>$.  Then, $I_{\Delta}$ has the primary-decomposition (see \cite[Theorem $1.7$]{MS}):
$$I_{\Delta} = \bigcap_{F\in\mathcal F(\Delta)} (x_i\mid i\notin F).$$

It follows that for $n\geqslant 1$, the $n$-th symbolic power of $I_{\Delta}$ is
$$I_{\Delta}^{(n)} = \bigcap_{F\in\mathcal F(\Delta)} (x_i\mid i\notin F)^n.$$

The {\it Alexander dual} of $\Delta$, denoted by $\Delta^*$, is the simplicial complex over $V$ defined by
$$\Delta^* = \{V\setminus \tau\mid \ \tau\notin \Delta\}.$$
Notice that $(\Delta^*)^* = \Delta$. If $I = I_\Delta$ then we shall denote the Stanley-Reisner ideal of the Alexander dual $\Delta^*$ by $I^*$. From the primary decomposition of $I_\Delta$ we obtain (see \cite[Definition 1.35]{MS}):
$$I_{\Delta}^* = (\x_{V\setminus F}\mid F\in \mathcal F(\Delta)).$$

The dual between homology groups of $\Delta$ and $\Delta^*$ is given by (see \cite[Theorem 5.6]{MS}). 

\begin{lem} \label{AlexanderDual} $\h_{i-1}(\Delta^*; K) \cong \h_{r-2-i}(\Delta; K)$ for all $i$.
\end{lem}

The regularity of a square-free monomial ideal can compute via the non-vanishing of reduced homology of simplicial complexes. From Hochster's formula on the Hilbert series of  the local cohomology module $H_\mi^i(R/I_\Delta)$ (see \cite[Theorem 13.13]{MS}), one has:

\begin{lem}\label{Hochster-Reg} For a simplicial complex $\Delta$, we have
$$\reg(I_\Delta)  = \max\{d\mid \h_{d-1}(\lk_{\Delta}(\sigma); K)\ne 0, \text{ for some } \sigma\in\Delta\} + 1.$$
\end{lem}

\subsection{Degree complexes} Let $I$ be a non-zero monomial ideal. Since $R/I$ is an $\mathbb N^r -$ graded algebra, $H^i_{\frak m}(R/I)$ is an $\mathbb Z^r$-graded module over $R/I$ for every $i$. For each degree $\alb=(\alpha_1,\ldots,\alpha_r)\in\Z^r$, in order to compute $\dim_K H_{\mi}^i(R/I)_{\alb}$ we use a formula given by Takayama \cite[Theorem $2.2$]{T} which is a generalization of Hochster's formula for the case $I$ is squarefree \cite[Theorem 4.1]{HO}.

Set $G_{\alb}:=\{i\mid \alpha_i<0\}$. For a subset $F\subseteq V$, we let $R_F:=R[x_i^{-1}\mid i\in F]$. Define the {\it degree complex} $\Delta_{\alb}(I)$ by
\begin{equation}\label{degree-complex}
\Delta _{\alpha }(I) :=\{F\subseteq V\setminus G_{\alb}\mid x^{\alpha }\notin IR_{F\cup G_{\alb}}\}.
\end{equation}

\begin{lem} \label{TA}\cite[Theorem 2.2]{T}  $\dim_K {H_{\frak m}^i(R/I)_{\alb}}=\dim_K \widetilde{H}_{i-\mid G_{\alpha }\mid-1 }(\Delta _{\alb}(I);K).$ 
\end{lem}

The following result of Minh and Trung is very useful for computing $\Delta_{\alb}(I_\Delta^{(n)})$,  for $\alb\in\N^r$ and $n\geqslant 1$.

\begin{lem}\cite[Lemma 1.3]{MT1} \label{MTr} Let $\Delta$ be a simplicial complex and $\alb \in\N^r$. Then,
	$$\F(\Delta_{\alb}(I_\Delta^{(n)})) =\left\{F\in \F(\Delta) \mid \sum_{i\notin F} \alpha_i \leqslant n-1\right\}.$$ 
\end{lem}

\subsection{Edge Ideals} 
Let $G$ be a finite simple graph. Assume that $V(G)=\{1,\ldots,r\}$. The edge ideal of $G$ is define by
$$I(G) = (x_ix_j \mid \{i,j\} \in E(G)) \subseteq R.$$

Let $\Delta(G)$ denote the set of all independent sets of $G$. Then, $\Delta(G)$ is a simplicial complex, called the {\it independence complex} of $G$. It is well-known that $I(G) = I_{\Delta(G)}$.
\medskip

H\`{a} and Van Tuyl  gave a bound for $\reg I(G)$ via the matching number of $G$.
\begin{lem} \cite[Theorem 6.7]{HTuyl} \label{upperBoundRegEdge} $\reg(I(G)) \leqslant \nu(G)+1$.\end{lem}

In the paper we need a characterization of graphs $G$ such that $\reg(I(G)) = \nu(G)+1$.  In order to do this, we recall the definition of Cameron-Walker graphs: a graph $G$ is called a {\it Cameron-Walker} graph if $\nu'(G)=\nu(G)$ (for the structure of such a graph, see  \cite{CaWa, HHKO}). Then, we have the folowing result (see \cite[Theorem 11]{Tr}).

\begin{lem}\label{Tr}  Let $G$ be a graph. Then, $\reg(I(G)) = \nu(G)+1$ if and only if each connected component of $G$ is either a pentagon or a Cameron-Walker graph.
\end{lem}

\subsection{Cover Ideals} Let $G = (V,E)$ be a graph on the vertex set $V = \{1,\ldots,r\}$. A {\it vertex cover} of $G$ is a subset of $V$ which meets every edge of $G$; a vertex cover is {\it minimal} if none of its proper subsets is itself a cover. The {\it cover ideal} of $G$ is defined by 
$$J(G) := (\x_{\tau} \mid \tau \text{ is a minimal vertex cover of } G).$$
It is well-known that the cover ideal $J(G)$ has the primary decomposition
\begin{equation} \label{intersect}
J(G) = \bigcap_{\{u,v\}\in E} (x_u, x_v).
\end{equation}
It follows that $J(G)^* = I(G)$.

From (\ref{intersect}), we have $J(G)$ is the Stanley-Reisner ideal corresponding with the simplicial complex
\begin{equation}\label{complex-cover}
\Delta(J(G)) = \left<V\setminus e \mid e\in E\right>.
\end{equation}

Lemma \ref{MTr} applies for $J(G)$ as follows.

\begin{lem} \label{uni-complex} For every $\alb=(\alpha_1,\ldots,\alpha_r)\in\N^r$ and $n\geqslant 1$, we have
$$
\Delta _{\alb}(J(G)^{(n)})=\left<V\setminus \{u,v\}\mid \{u,v\}\in E \text{ and }  \alpha_u+\alpha_v \leqslant n-1\right>.
$$
\end{lem}

It is worth mentioning that the cover ideal $J(G)$ of $G$ is normally torsion-free, i.e. $J(G)^{(n)} = J(G)^n$ for all $n\geqslant 1$, if and only if $G$ is bipartite (see \cite[Theorem 5.1]{HHT1}). In particular, $\sdst(J(G)) = \dst(J(G))$ for any bipartite graph $G$.
\medskip

In sequel, we need some facts about the behavior of the symbolic depth function of $J(G)$ (see \cite[Theorems 3.2 and 3.4]{HKTT}).

\begin{lem}\label{hktt} Let $G$ be a graph. Then,
\begin{enumerate}
\item The sequence $\{\depth R/J(G)^{(n)}\}_{n\geqslant 1}$ is decreasing, i.e.
$$\depth R/J(G) \geqslant \depth R/J(G)^{(2)} \geqslant \depth R/J(G)^{(3)} \geqslant \cdots$$
\item $\depth R/J(G)^{(n)} = r - \nu_0(G)-1$ for all $n\geqslant 2\nu_0(G)-1$.
\end{enumerate}
\end{lem}

As a consequence we obtain.
\begin{lem}\label{sdstability} Let $G$ be a graph. Then,
$$\sdst(J(G)) = \min \{n\geqslant 1\mid \depth R/J(G)^{(n)} \leqslant  r-\nu_0(G)-1\}.$$
\end{lem}

According to \cite{HV} we say that an ideal $I$ has {\it symbolic constant depth function} if $\sdst(I)=1$.  The following lemma gives a characterization of cover ideals which have symbolic constant depth functions in terms of regularity of edge ideals. 

\begin{lem} \label{constant-depth} Let $G$ be a graph. Then, $J(G)$ has constant depth function if and only if $\reg I(G) =\nu_0(G)+1$.
\end{lem}
\begin{proof} Since $I(G) = J(G)^*$, by Lemma  \cite[Theorem 5.59]{MS} we have $\pd R/J(G) = \reg I(G)$. Together with \cite[Theorem 1.3.3]{BH}, it yields 
\begin{equation}\label{depth-reg}
\depth R/J(G) = r - \reg I(G).
\end{equation}

Now assume that $\sdst(J(G)) = 1$, so that $\depth R/J(G) = r - \nu_0(G)-1$ by Lemma \ref{hktt}. Together with (\ref{depth-reg}) we get $\reg I(G) = \nu_0(G)+1$.

Conversely, assume that $\reg I(G) = \nu_0(G)+1$. Then, $\depth R/J(G) = r - \nu_0(G)-1$ by (\ref{depth-reg}), and so $\sdst(J(G))=1$ by Lemma \ref{sdstability},  as required.\end{proof}

%%%%%%%%%%%%%%%%%%%%%%%%%%%%%%%%%%%%%%%%%%%%%%%%%%%%%%%%%%%%%%%
\section{Upper bounds}

In this section we will establish bounds for $\sdst(J(G))$. We start with a result that allows us to bound $\dst(J(G))$ in terms of induced subgraphs of $G$.

\begin{lem}\label{subgraph} Let $H$ be an induced subgraph of $G$ such that $\nu_0(H) = \nu_0(G)$. Then, $$\sdst(J(G)) \leqslant \sdst(J(H)).$$  
\end{lem}
\begin{proof} Assume that $V(H) = \{1,\ldots,s\}$.  Let $S = K[x_1,\ldots,x_s]$ and $n=\sdst(J(H))$. By Lemma \ref{hktt} we have
\begin{equation}\label{U1}
\depth S/ J(H)^{(n)} = s - \nu_0(H)-1 = s - \nu_0(G)-1.
\end{equation}

Let $F = \{s+1,\ldots,r\}$. Then, $J(H) = J(G)R_F \cap S$. Together with Equality (\ref{U1}) with  \cite[Lemma 1.3]{HKTT} we have 
$$\depth R/J(G)^{(n)} \leqslant |F| + \depth S/J(H)^{(n)} = r-s + s - \nu_0(G)-1 = r-\nu_0(G)-1.$$
Thus, $\sdst(J(G)) \leqslant n$ by Lemma \ref{sdstability}, as required.
\end{proof}

\begin{lem}\label{sub-alpla} Let $G$ be a graph and an integer $n \geqslant \sdst(J(G))$. Then, there is an induced subgraph $H$ of $G$ which satisfies the following conditions:
\begin{enumerate}
\item $\nu_0(H) = \nu_0(G)$.
\item $\sdst(J(H)) \leqslant  n$.
\item Assume that $V(H) = [s]$. Then, there is $\alb = (\alpha_1,\ldots,\alpha_s)\in \N^s$ such that
$$\h_{s-\nu_0(H)-2}(\Delta_{\alb}(J(H)^{(n)}); K) \ne 0.$$
\end{enumerate}
\end{lem}
\begin{proof} By Lemma \ref{hktt} we have $\depth R/J(G)^{(n)} = r - \nu_0(G) -1$, and hence there is $\alb = (\alpha_1,\ldots,\alpha_r) \in \Z^r$ such that
$$H_{\mi}^{r - \nu_0(G) -1}(R/J(G)^{(n)})_{\alb} \ne \zv.$$

We may assume that $G_{\alb} = \{s+1,\ldots,r\}$ for $0\leqslant s\leqslant r$. By Lemma \ref{TA} we have
\begin{equation}\label{non-zero-1}
\h_{s-\nu_0(G)-2}(\Delta_{\alb}(J(G)^{(n)}); K) \ne \zv,
\end{equation}
since $r-\nu_0(G)-1 -|G_{\alb}| - 1 = s-\nu_0(G)-2$.

Let $F=G_{\alb}$ and $S = K[x_1, \ldots,x_s]$. Let $H$ be an induced subgraph of $G$ on the vertex set $[s]$. Then, $J(H) = J(G) R_F \cap S$. By \cite[Lemma 1.4]{HKTT}  we have 
\begin{align*}
\depth S/J(H)^{(n)}&=\depth R/J(G)^{(n)}-|F| = (r - \nu_0(G) -1) - (r - s) \\
&= s -\nu_0(G)-1.
\end{align*}
Together with Lemma $\ref{hktt}$ we obtain 
$$s -\nu_0(G)-1 = \depth S/J(H)^{(n)}\geqslant s -\nu_0(H)-1,$$
hence $\nu_0(G) \leqslant \nu_0(H)$. On the other hand, $\nu_0(H) \leqslant \nu_0(G)$ because $H$ is an induced subgraph of $G$, it forces $\nu_0(H) = \nu_0(G)$, and so
$$\depth S/J(H)^{(n)}= s -\nu_0(H)-1.$$
By Lemma \ref{sdstability} we have $\sdst(J(H))\leqslant n$.

Let $\alb' = (\alpha_1,\ldots,\alpha_s) \in\N^s$. By   Equation (\ref{degree-complex}) we have
$$\Delta_{\alb'}(J(H)^{(n)}) = \Delta_{\alb}(J(G)^{(n)}).$$
Together with (\ref{non-zero-1}) it yields
$$\h_{s-\nu_0(H)-2}(\Delta_{\alb'}(J(H)^{(n)}); K) \ne \zv,$$
and the proof is complete.
\end{proof}

\begin{lem}\label{comb-stab} Let $G$ be a graph with a perfect  matching $M$ and  $n$ a positive integer. Then, there is $\alb \in \N^{2s}$, where $s = |M|$, such that
$$\alpha_u + \alpha_v \leqslant n-1, \text{ for } uv \in M, \text{ and } \alpha_u + \alpha_v \geqslant n, \text{ for } uv \in E(G) \setminus M,$$
if and only if  $\nu_0(G)=s$ and $n\geqslant \sdst(J(G))$. 
\end{lem}
\begin{proof} We may assume that $G$ has the vertex set $V = \{1, \ldots,2s\}$. We have 
$$\Delta_{\alb}(J(G)^{(n)}) =\left<V(G)\setminus e\mid e\in M\right>,$$
by Lemma \ref{MTr}. Assume that $M =\{a_1b_1, \ldots, a_sb_s\}$. Then, $(\Delta_{\alb}(J(G)^{(n)})^* = \Delta(G')$, where $G'$ is the graph that consists of disjoint edges $$a_1b_1, \ldots, a_sb_s.$$
Thus,
$$\Delta(G') = \left<\{a_1\}, \{b_1\} \right> * \cdots * \left<\{a_s\}, \{b_s\} \right>,$$
and hence $\h_{s-1}(\Delta(G'); K) \ne 0$. Together with Lemma \ref{AlexanderDual}, it yields 
$$\h_{s-2}(\Delta_{\alb}(J(G)^{(n)}); K) \ne 0.$$
By Lemma \ref{TA}, we obtain $H_{\mi}^{s-1}(R/J(G)^{(n)})_{\alb} \ne \zv$, and therefore 
$$\depth R/J(G)^{(n)} \leqslant s-1.$$
On the other hand, by Lemma \ref{hktt} we have
$$\depth R/J(G)^{(n)} \geqslant 2s-\nu_0(G)-1,$$
so that $s-1\geqslant 2s-\nu_0(G)-1$. Thus, $\nu_0(G) \geqslant s$. Since $s=\nu(G)$, we have $\nu_0(G) \leqslant s$. It implies that $\nu_0(G) = s$. Therefore,
$$\depth R/J(G)^{(n)} \leqslant s-1 = 2s- \nu_0(G)-1,$$
and so $n\geqslant \sdst(J(G))$ by Lemma \ref{sdstability}.

\medskip

Conversely,  assume that $\nu_0(G) = s$ and $n\geqslant \sdst(J(G))$.  By Lemma \ref{sub-alpla}, we deduce that there is $\alb = (\alpha_1,\ldots,\alpha_{2s})\in \N^{2s}$ such that
\begin{equation}\label{F01}
\h_{2s-\nu_0(G)-2}(\Delta_{\alb}(J(G)^{(n)}), K) \ne \zv.
\end{equation}

Let $G'$ be a subgraph of $G$ consists of all edges $\{i,j\}$ of $G$ such that $\alpha_i+\alpha_j \leqslant n-1$. By Lemma $\ref{uni-complex}$ we have $\Delta_{\alb}(J(G)^{(n)}) = \left<V(G)\setminus \{i,j\} \mid \{i,j\} \in E(G')\right>$.  It follows that
$$\Delta(G') = (\Delta_{\alb}(J(G)^{(n)}))^*.$$
By Lemma \ref{AlexanderDual} we have $\h_{s-1}(\Delta(G'); K) \cong \h_{s-2}(\Delta_{\alb}(J(G)^{(n)}; K)$. Note that $2s-\nu_0(G)-2=s-2$, together with $(\ref{F01})$ we get $\h_{s-1}(\Delta(G'); K)\ne \zv$.

Hence,  $\reg(I(G'))\geqslant s + 1$ by Lemma \ref{Hochster-Reg}. Together with Lemma \ref{upperBoundRegEdge}, it forces $\reg(I(G')) = s+1$. This also implies that $s=\nu(G')$, and hence $G'$ has a perfect matching.

On the other hand, every connected component of $G'$ is  either a pentagon or a Cameron-Walker graph by Lemma \ref{Tr}. The former case is impossible since $G'$ has a perfect matching. It follows that $G'$ is a Cameron-Walker graph and so $\nu(G')=\nu'(G')$. Thus, $G'$ must be disjoint edges, and thus $E(G')$ is a perfect matching in $G$. Since $G$ has unique perfect matching by Lemma \ref{unique-matching}, we deduce that $E(G') = M$ as subsets of $E(G)$. Consequently,
$$\Delta_{\alb}(J(G)^{(n)}) = \left<V \setminus e\mid e \in M\right>,$$
and the lemma follows by Lemma \ref{MTr}.
\end{proof}

The following lemma plays a key role in our work. It gives an explicit solution for Lemma \ref{comb-stab} in the case $G$ is bipartite.

\begin{lem} \label{assignment} Let $G$ be a bipartite graph with a perfect ordered matching $M$. Assume that $M = \{\{i,s+i\}\mid i =1,\ldots,s\}$, where $s = |M|$, and  the free parameter set $X=\{1,\ldots,s\}$ and the partner set $Y =\{s+i\mid i=1,\ldots,s\}$ for $M$ form a bipartition of $G$.  Let $\ell_0(M) = 2k-1$. For each $i=1,\ldots,s$, if $\ell(s+i)=2k_i-1$, let
$$\alpha_i = k_i-1 \text{ and } \alpha_{s+i} = k - k_i.$$ 
Then,
\begin{enumerate}
\item $\alpha_i + \alpha_{s+i}  = k-1$ for $i=1,\ldots,s$. 
\item If $\{i,s+j\}\in E(G)$ with $i < j$, then $\alpha_i + \alpha_{s+j} \geqslant k$.
\end{enumerate}
\end{lem}

Before giving a proof of this lemma, we explain how to assign the value $\alpha_i$ for each vertex $i$ by working on the graph $G$ depicted in Figure $4$.

\begin{center}
  
\includegraphics[scale=0.5]{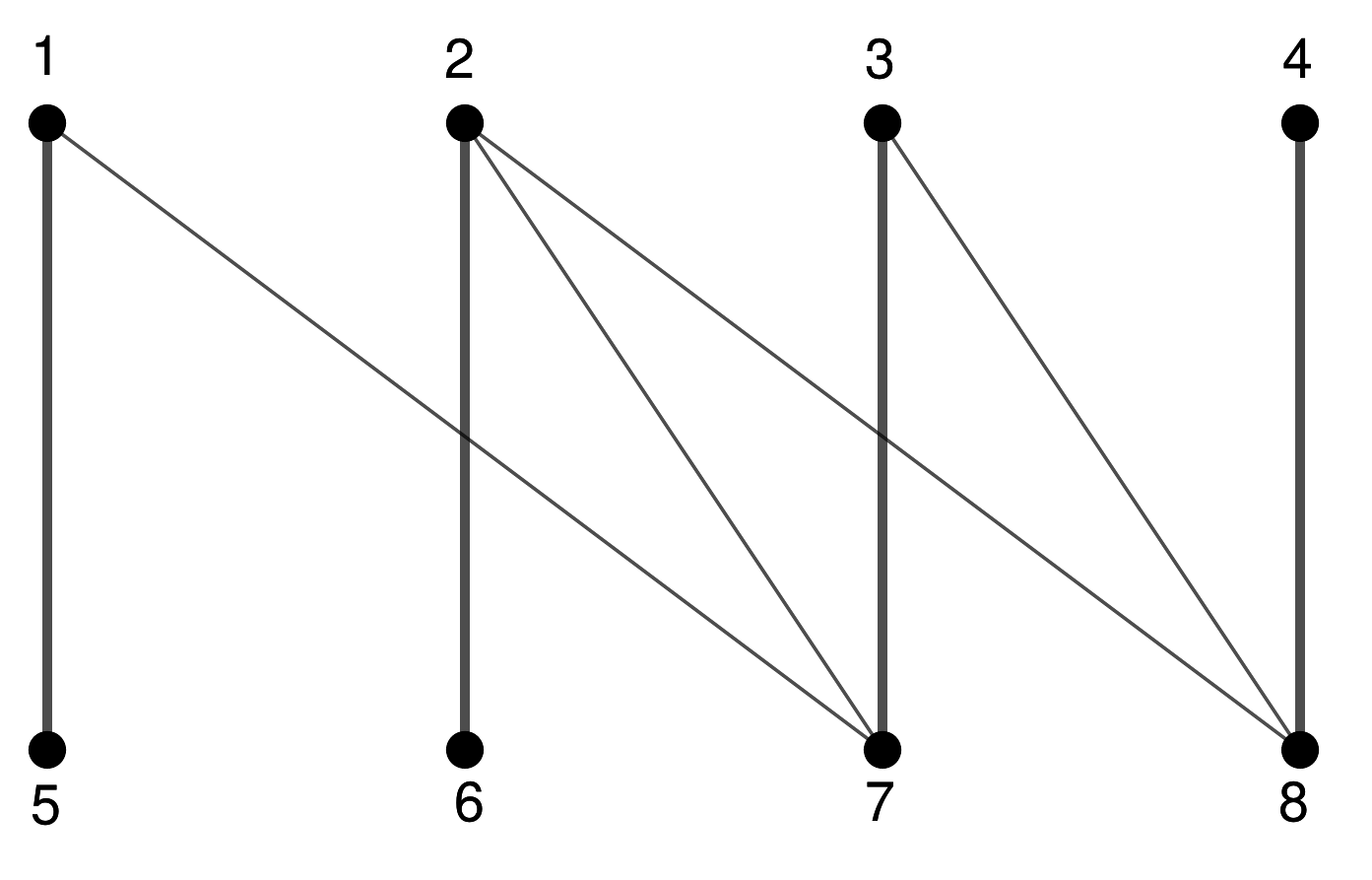}\\
\medskip

{\it Figure $3$. Assign values for vertices of a bipartite graph}
\end{center}

This graph has a perfect ordered matching  $M =\{\{1,5\}, \{2,6\}, \{3,7\}, \{4,8\}\}$. We can see that:
\begin{enumerate}
\item $5\to1\to 7 \to 3 \to 8\to 4$ is a longest $M$-admissible path of $5$, so  $k_1=3$ and $\alpha_1=2$.
\item $6\to2\to7\to3\to8\to4$ is a longest $M$-admissible path of $6$, so $k_2=3$ and $\alpha_2=2$.
\item $7\to 3 \to 8\to 4$ is a longest $M$-admissible path of $7$, so $k_3=2$ and $\alpha_3=1$.
\item $8\to 4$ is a longest $M$-admissible path of $8$, so $k_4=1$ and  $\alpha_4=0$.
\end{enumerate}
Hence, $\ell_0(M) = \max\{\ell(v) \mid v = 5,6,7,8\} = 5$, and hence $k=3$. For remain coordinates of $\alb$, we get
$$\alpha_5 = k -k_1 = 0, \ \alpha_6 = k-k_2=0,\  \alpha_7 = k-k_3 =1,\ \alpha_8 = k-k_4 = 2,$$ 
and so $\alb = (2,2,1,0,0,0,1,2)\in \N^8$.
\medskip

\begin{proof}[The proof of  Lemma \ref{assignment}]  $(1)$ is obvious, so we prove $(2)$. For an edge $\{i,s+j\}$ with $i < j$. If $p$ is an $M$-admissible path with the first vertex $s+j$ and the second vertex $j$. Then, the path $s+i, i, p$,  is an $M$-admissible path of length $\length(p) + 2$. It follows that $k_i \geqslant k_j + 1$. Thus,
$$\alpha_i + \alpha_{s+j} = (k_i-1) + (k-k_j)  = k + (k_i-k_j-1)\geqslant k,$$
as required.
\end{proof}

We are now in position to prove the  main result of the paper.

\begin{thm}\label{non-bipartite-theorem} $\sdst(J(G)) \leqslant (\ell(G)+1)/2$ for every graph $G$.
\end{thm}
\begin{proof} Let $M$ be a maximum ordered matching in $G$ with $\ell(M) = \ell(G)$.  We may assume that $M = \{\{i,s+i\}\mid i = 1, \ldots,s\}$ where $s =\nu_0(G) = |M|$, and $A= \{1,\ldots, s\}$ and $B = \{s+1,\ldots,2s\}$ are the free parameter and the partner sets for $M$, respectively. 

Then, $M$ is a maximum perfect ordered matching in $G(A,B)$. Let $\ell_0(M) = 2k-1$. By Lemma \ref{assignment}, there is a vector $\alb = (\alpha_1,\ldots, \alpha_{2s}) \in \N^{2s}$ such that for all $j=1,\ldots,s$, we have
\begin{equation}\label{EN1}
\alpha_j + \alpha_{s+j} = k -1 \text{ for all } j \in A,
\end{equation}
and
\begin{equation}\label{EN2}
\alpha_j + \alpha_{s+e} \geqslant k \text{ for all } \{j, s+e\} \in E[G(A,B)] \text{ with } j < e.
\end{equation}

Let $n = (\ell(G)+1)/2$.  For $j=1,\ldots,s$, we define
\begin{equation}\label{EN3}
\beta_j = \alpha_j \text{ and } \beta_{s+j} = n -k +\alpha_{s+j}.
\end{equation}
Note that $\btb\in \N^{2s}$ since $n\geqslant k$ by Lemma \ref{lem-ell}.

We now prove that for every edge $uv$ of $G[M]$, the following holds
\begin{equation}\label{EN4}
\beta_u + \beta_v = n-1, \text { if } uv \in M,
\end{equation}
and 
\begin{equation}\label{EN5}
\beta_u + \beta_v \geqslant n, \text{ for } uv \notin M.
\end{equation}

Indeed, if $uv$ in $M$, then $u\in A$ and $v=s+u$. Thus, by (\ref{EN1}) and (\ref{EN3}) we obtain
$$
\beta_u + \beta_v = \alpha_u + n-k+\alpha_{s+u} = n-k+\alpha_u + \alpha_{s+u}= n-k + k-1 =n-1,
$$
and (\ref{EN4}) follows.

In order to prove $(\ref{EN5})$, let $uv$ be an edge of $G[M]$ such that $uv\notin M$. We distinguish the following cases:
\medskip

{\it Case 1}: $uv$ has an end in $A$ and another one in $B$. We may assume that $u\in A$ and $v\in B$. In this case, $v = s + e$ for some $e\in A$ with $u < e$. From $(\ref{EN2})$ and ($\ref{EN3}$) we have
\begin{align*}
\beta_u + \beta_v &= \alpha_u + (n-k+\alpha_{s+e}) = n-k+(\alpha_u + \alpha_{s+e})\\
&\geqslant n-k + k =n,
\end{align*}
and (\ref{EN5}) holds true in this case.
\medskip

{\it Case 2}: $u,v\in B$. In this case, $u = s + e$ and $v = s + f$ for $e,f \in A$. Note that $\ell_1(M)\leqslant \ell(M)$ by Lemma \ref{lem-ell}.  Together with  Lemma \ref{assignment} we have
\begin{align*}
\alpha_e +\alpha_f &= \frac{\ell(s+e)-1}{2}+ \frac{\ell(s+f)-1}{2} =\frac{\ell(s+e)+\ell(s+f)+2}{2}-2\\
&\leqslant \frac{\ell_1(M)+1}{2}-2 \leqslant  \frac{\ell(M)+1}{2}-2 = n-2,
\end{align*}
so that $\alpha_e +\alpha_f\leqslant n-2$. Together with (\ref{EN3}), we have
\begin{align*}
\beta_u + \beta_v &= \beta_{s+e} + \beta_{s +f} = (n-k +\alpha_{s+e}) +( n-k+\alpha_{s+f})\\
&= 2n -2k+\alpha_{s+e}+\alpha_{s+f} = 2n-2k+(k-1-\alpha_e)+(k-1-\alpha_f)\\
&= 2n-2-(\alpha_e+\alpha_f) \geqslant n,
\end{align*}
 and Inequality (\ref{EN5}) follows.
 
 \medskip
 
 We now prove the theorem. Together $(\ref{EN4})$ and $(\ref{EN5})$ with Lemma \ref{comb-stab} we have $\sdst(J(G[M])) \leqslant n$. Finally, by Lemma \ref{subgraph} we obtain  $$\sdst(J(G)) \leqslant \sdst(J(G[M])) \leqslant n,$$  and the proof is complete.
\end{proof}

The next result shows that the bound in Theorem \ref{non-bipartite-theorem} is the true value for $\sdst(J(G))$ in the case the graph  has a perfect ordered matching.

\begin{prop} \label{PM} If a graph $G$ has a perfect ordered matching $M$, then $$\sdst(J(G)) = (\ell(G)+1)/2.$$
\end{prop}
\begin{proof} Let $s = \nu_0(G)$ and $n = \sdst(J(G))$. Let $A$ and $B$ be the free parameter and partner sets for $M$, respectively. By Theorem \ref{non-bipartite-theorem} we have $n \leqslant (\ell(G)+1)/2$, so it remain to prove that $n \geqslant (\ell(G)+1)/2$, or equivalently $n\geqslant (\ell(M)+1)/2$ by Remark \ref{perfect-ell}.

We first claim that $n\geqslant (\ell_0(M)+1)/2$. In order to prove this, by Lemma \ref{comb-stab} there is $\alb = (\alpha_1,\ldots,\alpha_{2s})\in \N^{2s}$ such that
\begin{equation}\label{delta-6}
\alpha_u + \alpha_v \leqslant n-1 \text{ for  } uv\in M, \ \text{ and } \alpha_u + \alpha_v \geqslant n \text{ for  } uv\in E(G)\setminus M.
\end{equation}

Assume $\ell_0(M) = 2k-1$ so that $(\ell_0(M)+1)/2=k$. Let $$u_1v_1, v_1u_2,u_2v_2, \ldots, u_kv_k$$
be an $M$-admissible path of length $2k-1$ for $u_1\in B$. Then, $u_iv_i \in M$ for $i=1,\ldots, k$ and $v_iu_{i+1}\notin M$ for $i=1,\ldots,k-1$. For every $i=1,\ldots, k-1$, by (\ref{delta-6}) we have 
$$\alpha_{v_i} + \alpha_{u_{i+1}} \geqslant n, \text{ and } \alpha_{v_{i+1}}+\alpha_{u_{i+1}} \leqslant n-1,$$ 
and therefore $\alpha_{v_i} > \alpha_{v_{i+1}}$. From the sequence
$$\alpha_{v_1} > \alpha_{v_2} >\cdots > \alpha_{v_k}$$
we deduce that $\alpha_{v_1} \geqslant k-1$.

Together with the fact that $\alpha_{u_1} + \alpha_{v_1} \leqslant n-1$, it yields $k-1\leqslant n-1$, or equivalently $n\geqslant k$, as claimed.

It remains to prove $n \geqslant (\ell_1(M)+1)/2$, or equivalently $\ell_1(M)\leqslant 2n-1$. Assume that $\ell_1(M) = \ell(u)+\ell(v)+1$ for some $uv \in G[B]$. Let $e,f\in A$ such that $ue, vf\in M$. By the same argument as in  previous paragraph, we obtain $\alpha_e \geqslant (\ell(u)-1)/2$ and $\alpha_f \geqslant (\ell(v)-1)/2$. Together with (\ref{delta-6}) we have
\begin{align*}
\ell_1(M) &= \ell(u)+\ell(v)+1\leqslant (2\alpha_e+1 +2\alpha_f+1)+1\\
&=2(\alpha_e+\alpha_f) +3 \leqslant 2(n - 1-\alpha_u +n-1-\alpha_v)+3 = 2(2n - (\alpha_u +\alpha_v))-1.
\end{align*}

On the other hand, since $uv\notin M$, by (\ref{delta-6}) we have $\alpha_u + \alpha_v \geqslant n$. Therefore,
$\ell_1(M) \leqslant 2(2n - n)-1 =2n-1$, and the lemma follows.
\end{proof}

It is known that $\sdst(J(G)) \leqslant 2\nu_0(G)-1$ (see \cite[Theorem 3.4]{HKTT}). Since $\ell(G) \leqslant 4\nu_0(G)-3$ by Lemma \ref{up-nu}, thus the bound in Theorem \ref{non-bipartite-theorem} improves it. The following example shows that our bound may be much more less than the previous bound.

\begin{exam} Let $s \geqslant 1$ and let $G$ be a graph with the vertex set $V(G) =\{1,2, \ldots, 4s\}$ and the edge set (see Figure 5)
$$E(G) = E_1 \cup E_2 \cup \{\{2s+1, 3s+1\}\},$$
where
$$E_1 = \{\{i,2s+j\}, \{p,q\} \mid 1\leqslant i \leqslant j \leqslant s, \text{ and } 2s+1\leqslant p < q \leqslant 3s\},$$
and
$$E_2 = \{\{i,2s+j\}, \{p,q\} \mid s+1\leqslant i \leqslant j \leqslant 2s, \text{ and } 3s+1\leqslant p < q \leqslant 4s\}.$$
Then, we have $\nu_0(G) = 2s$, $\ell(G) = 4s-1$ and $\sdst(J(G)) = 2s$.
\begin{center}

\includegraphics[scale=0.6]{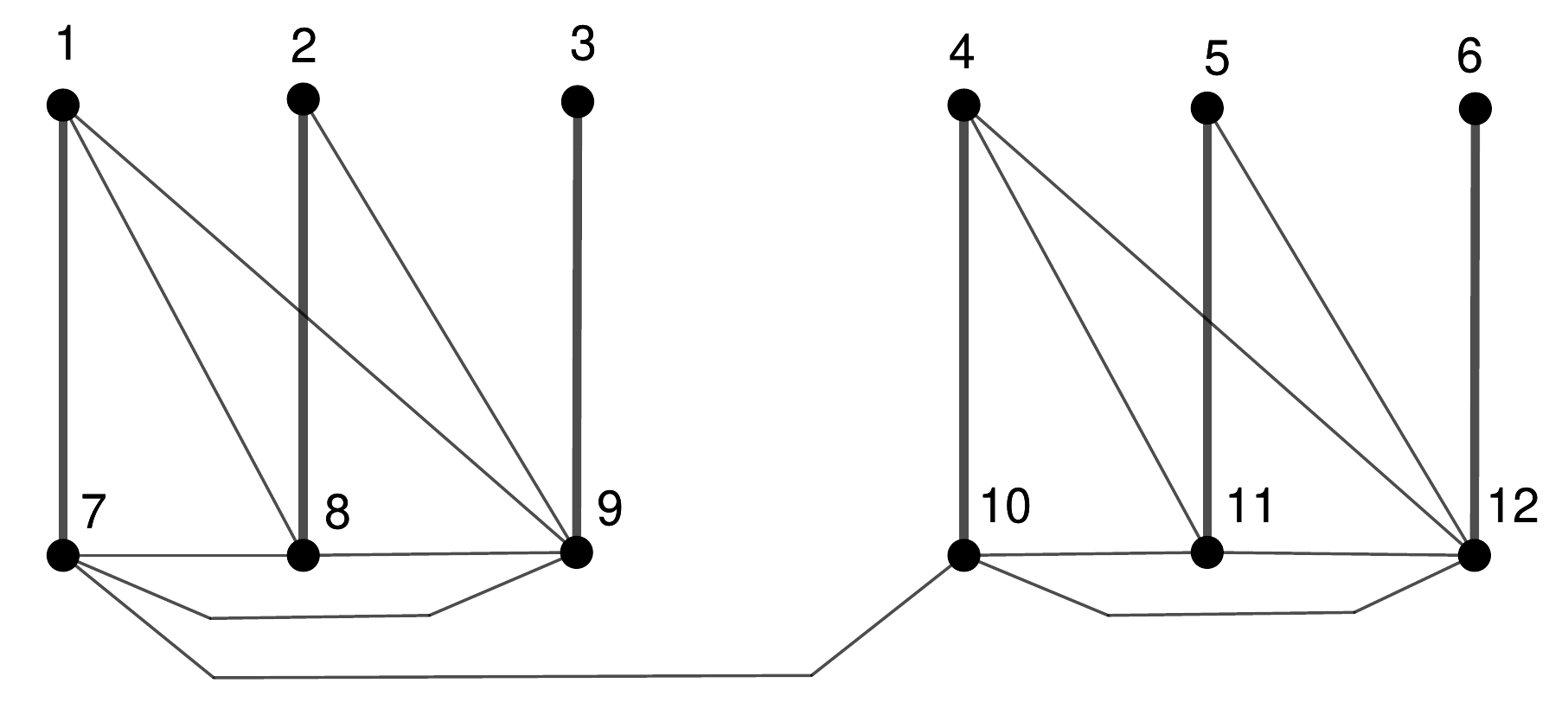}\\
\medskip

{\it Figure $4$. The graph $G$ with $s=3$}
\end{center}

\end{exam}
\begin{proof} Clearly $M = \{\{i, 2s+i\}\mid i = 1,\ldots, 2s\}$ is a perfect ordered matching in $G$. In particular, $\nu_0(G) = 2s$. 

We first claim that $\ell(G) = 4s-1$. Note that $A = \{1,\ldots,2s\}$ and $B = \{2s+1, \ldots,4s\}$ are the free parameter and the partner sets for $M$, respectively. Let 
$$A_1=\{1,\ldots, s\}, \ B_1 = \{2s+1, \ldots, 3s\}, \text{ and } M_1 = \{\{i,2s+i\}\mid j =1,\ldots,s\}$$
and 
$$A_2=\{s+1,\ldots, 2s\}, \ B_1 = \{3s+1, \ldots, 4s\}, \text{ and } M_2 = \{\{i,2s+i\}\mid j =s+1,\ldots,2s\}.$$

Then, each $M_i$ is a perfect ordered matching in $G(A_i, B_i)$ with free parameter set $A_i$ and partner set $B_i$. 

Since every $M$-admissible path in $G(A,B)$ is an $M_i$-admissible path in $G(A_i,B_i)$ for some $i=1,2$. It follows that the length of an $M$-admissible path in $G(A,B)$ is at most $2s-1$, and so $\ell_0(M) \leqslant 2s-1$.

On the other hand, observe that $$2s+1,1,2s+2,2, \ldots, 3s,s$$ is an $M$-admissible path of $2s+1$, we have $\ell(2s+1) = 2s-1$ and $\ell_0(M) = 2s-1$. By the same way, $\ell(3s+1) = 2s-1$. It follows that
$\ell_1(M)=4s-1$, and so $\ell(M) = 4s-1$. Thus, $\ell(G) = 4s-1$ by Remark \ref{perfect-ell}.

Since $G$ has a perfect ordered matching, by Proposition \ref{PM} we obtain \break $\sdst(J(G)) = (\ell(G)+1)/2 = 2s$, as required.
\end{proof}

If $G$ is bipartite, then $J(G)^n=J(G)^{(n)}$ for all $n\geqslant 1$ by \cite[Theorem 5.1]{HHT1}, so that $\sdst(J(G)) = \dst(J(G))$. Together with Theorem \ref{non-bipartite-theorem} we obtain:

\begin{cor}\label{bipartite-theorem} If $G$ is a bipartite graph, then $\dst(J(G)) \leqslant (\ell(G)+1)/2$.
\end{cor}

We conclude this section with an example taken from \cite{DK} which shows that it is hopeless to give a combinatorial formula of $\dst(J(G))$ for bipartite graph $G$, and also is $\sdst(J(G))$. Because these indices actually depend on the characteristic of the base field $K$.

\begin{exam}(see \cite[Example 4.8]{DK}) \label{char-depend} Consider the graph $G$ with the vertex set 
$$V = \{x_1,\ldots, x_{10}, y_1,\ldots,y_6\},$$ 
and the edge set
\begin{align*}
E =&\{x_1y_1,x_2y_1,x_3y_1,x_7y_1,x_9y_1,x_1y_2,x_2y_2,x_4y_2,x_6y_2,x_{10}y_2,\\
&x_1y_3,x_3y_3,x_5y_3,x_6y_3,x_8y_3,x_2y_4,x_4y_4,x_5y_4,x_7y_4,x_8y_4,\\
&x_3y_5,x_4y_5,x_5y_5,x_9y_5,x_{10}y_5,x_6y_6,x_7y_6,x_8y_6,x_9y_6,x_{10}y_6\}.
\end{align*}
Let $J=J(G)$ be the cover ideal of  $G$ in $R = K[x_1,\ldots,x_{10},y_1,\ldots,y_6]$. Then,
$$
\dst(J) = \begin{cases} 2 &\text{ if } \ch(K) = 0,\\
1 &\text{ if } \ch(K) = 2.
\end{cases}
$$

Before giving the proof, note that $G$ has a bipartition $(\{x_1,\ldots,x_{10}\}, \{y_1,\ldots,y_6\})$, so that $\nu(G) \leqslant 6$.  Hence, $\ell(G)\leqslant 2\cdot 6-1=11$ by Lemma \ref{up-nu} (2). Now in order to prove the result above, we use Macaulay2 (see \cite{GS}) to compute the initial values of the depth function $\depth R/J^k$ for $k=1,\ldots,6$, and then together with Corollary \ref{bipartite-theorem} we obtain the desired result. Namely,

\medskip

{\it Case 1}: $\ch(K) = 0$. By using Macaulay2 we get 
$$\depth(R/J) = 12 \text{ and } \depth R/J^k = 11 \text{ for } k=2,\ldots,6.$$ 
Note that $(\ell(G)+1)/2 \leqslant 6$. With this data, by Theorem \ref{bipartite-theorem} we have $\depth R/J^k = 11$, for $k\geqslant 6$. Thus, $\dst(J) = 2$.

\medskip

{\it Case 2}: $\ch(K) = 2$. By using Macaulay2 again we get $\depth R/J^k = 11$ for $k=1,\ldots,6$. By the same as in the previous case, $\depth R/J^k = 11$, for $k\geqslant 6$. Thus, $\dst(J) = 1$.
\end{exam}

\section{Cover Ideals of Forests and Cycles}

In this section we extend the formula for $\sdst(J(G))$ in Proposition \ref{PM} for some classes of graphs. Then, we  compute $\sdst(J(G))$ explicitly in the case $G$ is a path or a cycle. We start with the following  technique lemma.

\begin{lem} \label{EQM} Let $G$ be a graph with $\nu(G) = \nu_0(G)$. If  $G$ contains no pentagons,  then $\sdst(J(G)) = \sdst(J(H))$, where $H$ is an induced subgraph of $G$ with $\nu_0(H) = \nu_0(G)$ and it has a perfect ordered matching.
\end{lem}
\begin{proof}  Let $n = \sdst(J(G))$. By Lemma \ref{sub-alpla}, there is an induced subgraph $H$ of $G$ such that $\nu_0(H)=\nu_0(G)$, $n \geqslant \sdst(J(H)$, and if we assume $V(H) = \{1,\ldots,s\}$ then there is $\alb \in \N^s$ such that
\begin{equation}\label{non-zero}
\h_{s-\nu_0(H)-2}(\Delta_{\alb}(J(H)^{(n)}); K) \ne \zv.
\end{equation}
Note that $n \leqslant \sdst(J(H))$ by Lemma \ref{subgraph} so  $n =\sdst(J(H))$. 

Let $G'$ be the subgraph of $H$ with the edge set determined by 
$$E(G') = \{uv \mid uv\in E(H) \text{ and  } \alpha_u + \alpha_v \leqslant n-1\}.$$
Since $\Delta_{\alb}(J(H)^{(n)}) = \{[s] \setminus e\mid e \in E(G')\}$ and this simplicial complex is not a cone due to (\ref{non-zero}), it implies that that $V(G') = V(H)$.

Since $(\Delta_{\alb}(J(H)^{(n)}))^* = \Delta(G')$, together with (\ref{non-zero}) and Lemma \ref{AlexanderDual} this fact gets
$$\h_{\nu_0(H)-1}(\Delta(G'); K) \ne \zv.$$
Therefore, $\reg I(G') \geqslant \nu_0(H) + 1$ by Lemma \ref{Hochster-Reg}.

Since $G'$ is a subgraph of $H$, $\nu(G') \leqslant \nu(H)=\nu_0(H)$. By Lemma \ref{upperBoundRegEdge} one has
$$\reg(I(G')) \leqslant \nu(G')+1\leqslant \nu(H)+1 = \nu_0(H)+1,$$
and hence $\nu(G') = \nu_0(H)$ and $\reg(I(G')) = \nu(G')+1$. 

By Lemma \ref{Tr}, every connected component of $G'$ is either a pentagon or a Cameron-Walker graph. Since $G$ contains no pentagons,  $G'$ must be a Cameron-Walker graph, and hence $\nu'(G') = \nu(G')$. Since $\nu(G') = \nu_0(H) = \nu_0(G)=\nu(G)$, it follows that  $G'$ has an induced matching, say $M$, with $|M| = \nu(G)$.

We next show that
\begin{equation}\label{LQ}
\alpha_u + \alpha_v \leqslant n-1 \text{ for  } uv\in M, \ \text{ and } \alpha_u + \alpha_v \geqslant n \text{ for  } uv\in E(G[M])\setminus M.
\end{equation}
Indeed, for any edge $uv$ of $G$, if $uv\in M$ then $uv\in E(G')$. By definition of $G'$ we have, $\alpha_u + \alpha_v \leqslant n-1$. Assume that $uv$ is not in $M$. Since $M$ is an induced matching in $G'$, the edge $uv$ has no common vertex with any edge in $M$. It implies that $uv$ is not an edge of $G'$, so $\alpha_u+\alpha_v \geqslant n$, and (\ref{LQ}) follows.

Together (\ref{LQ}) with Lemma \ref{comb-stab} we obtain $\nu_0(G[M]) = |M| = \nu_0(G)$ and $n\geqslant \sdst(J(G[M])$. On the other hand, since $G[M]$ is an induced subgraph of $G$ with $\nu_0(G[M]) =\nu_0(G)$, one has $n\leqslant \sdst(J(G[M])$ by Lemma \ref{subgraph}. Hence,
$n = \sdst(J(G[M])$, and the lemma follows.
\end{proof}

\begin{prop} \label{EPM} Let $G$ be a graph with $\nu(G) = \nu_0(G)$ such that it contains no pentagons. Then,  $\sdst(J(G))=(\ell(G)+1)/2$.
\end{prop}
\begin{proof} By Theorem \ref{non-bipartite-theorem}, we have $\sdst(J(G)) \leqslant (\ell(G)+1)/2$. It remains to prove that $\sdst(J(G)) \geqslant (\ell(G)+1)/2$. In order to prove this, by Lemma \ref{EQM}, there is an induced subgraph $H$ of $G$ such that $\nu_0(G)=\nu_0(H)$, $\sdst(J(G))=\sdst(J(H))$ and $H$ has a perfect ordered matching, say $M$.

By Proposition \ref{PM} we have $\sdst(J(H)) = (\ell(M)+1)/2 \geqslant (\ell(G)+1)/2$, and  proposition follows.
\end{proof}

\begin{cor}\label{bipartite-cor} If $G$ is a bipartite graph with $\nu(G) = \nu_0(G)$, then
$$\dst(J(G))  = (\ell(G)+1)/2.$$
\end{cor}
\begin{proof} Since $G$ is bipartite, it contains no pentagons. The corollary follows from Proposition \ref{EPM}.
\end{proof}

\begin{prop}\label{main-cor} If $G$ is a forest, then $\dst(G)  = (\ell(G)+1)/2$.
\end{prop}
\begin{proof} Since every connected component of $G$ is a tree, we deduce that $\nu(G) = \nu_0(G)$ by \cite[Proposition 3.10]{CV}. Hence, the proposition  follows from Corollary \ref{bipartite-cor}.
\end{proof} 

The rest of the paper is devoted to compute explicitly the stability index of depth functions for the cover ideals of paths and cycles. We first start with paths.

\begin{exam} \label{exm1} Let $P_r$ be the path with $r$ vertices. Then,
$$
\dst(J(P_r))=
\begin{cases}
\frac{r}{2} & \text{ if } r \text{ is even},\\
\left\lceil \frac{r-1}{4}\right\rceil & \text{ if } r \text{ is odd}.
\end{cases}
$$
\end{exam}
\begin{proof} Note that $P_r$ has $r-1$ edges, and $\nu_0(P_r)=\nu(P_r) = \left\lfloor r/2\right\rfloor$.

If $r$ is even, then $P_r$ is a bipartite graph has a perfect ordered matching. Let $M$ be a perfect ordered matching in $P_r$. Since $P_r$ is an $M$-alternating path with the first and the last edges in $M$, hence $\ell(P_r) = \length(P_r) = r-1$. By Proposition \ref{main-cor} we have $\dst(J(P_r))=r/2$.

If $r$ is odd, then $\nu_0(P_r) = \nu(P_r) = (r-1)/2$. In order to compute $\ell(P_r)$, let $H$ be an induced subgraph of $P_r$ such that $\nu_0(H) = \nu_0(P_r)$ and $H$ has a perfect ordered matching. Since $\nu_0(H) = (r-1)/2$, we have $|V(H)| =r-1$, thus $H$ is obtained from $P_r$ by deleting just one vertex. Hence, $H$ is disjoint union of two disjoint paths, say $P_s$ and $P_t$. Since each $H$ has a perfect ordered matching, so are these paths. In particular, $s$ and $t$ are even, and therefore 
$$\ell(H) = \max\{\ell(P_s), \ell(P_t)\}.$$
By the previous case, we have
$$\ell(H) = \max\left\{s-1, t-1\right\} =\max\left\{s, t\right\}-1.$$
Therefore,
\begin{align*}
\ell(P_r) &= \min\left\{ \max\{s,t\}-1\mid s + t =r-1, s \text{ and  } t \text{ are even}  \right\}\\
&=2\min\{ \max\{a,b\}\mid a,b\in\N \text{ and } a + b = (r-1)/2 \} -1\\
&= 2\left\lceil \frac{r-1}{4}\right\rceil-1.
\end{align*}
By Proposition \ref{main-cor},  $\dst(J(P_r))=\left\lceil (r-1)/4\right\rceil$, and the proof is complete.
\end{proof}

Next we compute $\sdst(C_r)$ for the case $r$ is odd.

\begin{exam} \label{exm2} Let $C_r$ be an odd cycle. Then,
$$
\sdst(J(C_r))=
\begin{cases}
1 & \text{ if } r = 5.\\
\frac{r-1}{2} &\text{ otherwsie}.
\end{cases}
$$
\end{exam}
\begin{proof} Assume that $r\ne 5$. Since $r$ is odd, we have $\nu(C_r) = \nu_0(C_r)=(r-1)/2$. By Lemma \ref{EQM}, there is an induced subgraph $H$ of $C_r$ with $\nu_0(H) = \nu_0(C_r)$ such that $H$ has a perfect ordered matching and $\sdst(J(C_r)) = \sdst(J(H))$.

Since $H$ is an induced subgraph of $C_r$ with $r-1$ vertices, it is obtained by deleting one vertex from $C_r$, and thus $H$ is a path with $r-1$ vertices. By Example \ref{exm1} we get $\sdst(J(C_r)) = \sdst(J(H)) = (r-1)/2$.

In the case $r=5$, we have $\nu_0(G) = 2$. By \cite[Theorem 5.2]{BHT} we get $$\reg I(C_5) = \left\lfloor 5/3\right\rfloor + 2 = 3 = \nu_0(G)+1.$$
Thus, $\sdst(J(C_5)) = 1$, by Lemma \ref{constant-depth}.
\end{proof}

Finally we compute $\dst(J(C_r))$ for the case $r$ is even. In this case, the results show that $\dst(J(C_r)) = (\ell(C_r)+1)/2$ except for $r=8$.

\begin{exam} \label{exm3} Let $C_r$ be an even cycle. Then, $\ell(C_r) = 2\lceil (r-2)/4\rceil-1$ and
$$\dst(J(C_r)) =
\begin{cases}
1 &\text{ if } r = 8,\\
\left\lceil \dfrac{r-2}{4}\right\rceil &\text{ if } r \ne 8. 
\end{cases}
$$
\end{exam}
\begin{proof} The equality $\ell(C_r) = 2\lceil (r-2)/4\rceil-1$ is proved by the same argument as in Example \ref{exm1}. We now compute $\dst(J(C_r))$. First note that $\nu_0(C_r)  = r/2-1$, $\nu(C_r) = r/2$, and by \cite[Theorem 5.2]{BHT}, one has
\begin{equation}\label{reg-Cr}
\reg I(C_r)=\begin{cases}
\lfloor r/3\rfloor +1 &\text{ if } r \equiv 0,1 \pmod 3,\\
\lfloor r/3\rfloor +2 &\text{ if } r \equiv 2 \pmod 3.
\end{cases}
\end{equation}

Together with Lemma \ref{constant-depth}, these facts yield $\dst(J(C_r)) =1$ for $r=4,6,8$. Therefore, we may assume that $r \geqslant 10$.

Let $n = \dst(J(C_r))$. We first claim that $C_r$ has  a proper induced subgraph $H$ such that $\nu_0(H) = \nu_0(G)$ and $\dst(J(C_r)) = \dst(H)$. 

Indeed, assume on the contrary that this is not the case. Then, by Lemma \ref{sub-alpla} there is $\alb = (\alpha_1,\ldots,\alpha_r)\in\N^r$ such that
\begin{equation} \label{EQ100}
\h_{r-\nu_0(C_r)-2}(\Delta_{\alb}(J(C_r)^{(n)}); K) \ne \zv.
\end{equation}

Let $G$ be a subgraph of $C_r$ determined by the edge set 
$$E(G) = \{uv \mid uv\in E(C_r) \text{ and }  \alpha_u+\alpha_v \leqslant n-1\}.$$
Since $\Delta_{\alb}(J(C_r)^{(n)})$ is not a cone by (\ref{EQ100}), it implies that $V(G) = V(C_r)$.

Note that $\Delta(G) = (\Delta_{\alb}(J(C_r)^{(n)}))^*$, hence by Lemma \ref{AlexanderDual} and (\ref{EQ100}) we have
$$\h_{\nu_0(C_r)-1}(\Delta(G); K) \ne 0.$$
In particular, by Lemma \ref{Hochster-Reg} one has
\begin{equation} \label{EQ101}
\reg(I(G)) \geqslant \nu_0(C_r)+1.
\end{equation}\

We consider two possible cases:

\medskip

{\it Case 1}: $G=C_r$. Note that $\nu_0(G) = r/2-1$, so by (\ref{reg-Cr}) and (\ref{EQ101}) we have
$$\lfloor r/3\rfloor +k\geqslant r/2$$
where $k = 1$ if $r\equiv 0,1 \mod 3$ and $k = 2$ if $r\equiv 2\mod 3$. But this inequality is not true as $r\geqslant 10$,  hence this case is impossible.

\medskip

{\it Case 2}: $G \ne C_r$, i.e. $G$ is the graph obtained by removing some edges from $C_r$. In particular, $G$ is a forest, so $\reg I(G) = \nu'(G)+1$ by \cite[Theorem 2.18]{Z}.  Together with (\ref{EQ101}), it yields $\nu'(G)\geqslant \nu_0(C_r)$. It follows that $G$ has an induced matching $M$ with $|M| = \nu_0(C_r)$. Let $H = C_r[M]$. Observe that for each edge $uv$ of $C_r$ where $uv\notin M$, since $M$ is an induced matching in $G$, we have $uv\notin E(G)$, and hence
$\alpha_u + \alpha_v \geqslant n$. In particular,
$$\alpha_u +\alpha_v \leqslant n-1 \text{ for } uv\in M, \ \text{ and } \alpha_u+\alpha_v \geqslant n \text{ for } uv\in E(H) \setminus M.$$
Therefore, $\dst(J(H)) \leqslant n$ and $\nu_0(H) =|M|= \nu_0(C_r)$ by Lemma \ref{comb-stab}.  On the other hand, by Lemma \ref{subgraph} we obtain $n \leqslant \dst(J(H))$, and thus $n = \dst(J(H))$.  But the existence of the graph $H$ contradicts our assumption, and the claim follows.

\medskip

We now prove that $n= \lceil (r-2)/4\rceil$.  For any vertex $v$ of $C_r$,  $C_r \setminus v$ is a path with $r-1$ vertices and $\nu_0(C_r\setminus v) = (r-2)/2=\nu_0(C_r)$. Together with Lemma \ref{subgraph} and Example \ref{exm1}, it yields $n\leqslant \dst(C_r\setminus v) = \lceil (r-2)/4\rceil$.

On the other hand, by the claim above, $C_r$ has a proper induced subgraph $H$  such that $\nu_0(C_r)= \nu_0(H)$ and $n = \dst(J(H))$. Let $v$ be a vertex of $C_r$ which is not a vertex of $H$. Then, $H$ is an induced subgraph of $C_r\setminus v$. By Lemma \ref{subgraph} and Example \ref{exm1} we get
$$n = \dst(J(H)) \geqslant \dst(C_r\setminus v) = \lceil (r-2)/4\rceil.$$
It follows that $n= \lceil (r-2)/4\rceil$, and the proof is complete.
\end{proof}

\subsection*{Acknowledgment}  This work is  supported by International Centre of Research and Postgraduate Training in Mathematics (ICRTM), Institute of Mathematics, VAST under the grant number ICRTM04-2021.06.


\begin{thebibliography}{99}

\bibitem {BHT} S. Beyarslan, H.T. H\`{a} and T.N. Trung, {\it Regularity of powers of forests and cycles}, J. Algebraic Combin. {\bf 42} (2015), no. 4, 1077-1095. 

\bibitem{BM} J. A. Bondy, U. S. R. Murty, Graph Theory,  Graduate Texts in Mathematics,  {\bf 244},  Springer, New York 2008.

\bibitem {B} M. Brodmann, {\it The Asymptotic Nature of the Analytic Spread}, Math. Proc. Cambridge Philos Soc., {\bf 86}(1979), 35-39. 

\bibitem{BH} W.  Bruns and J. Herzog, Cohen–Macaulay rings, revised edition, Cambridge Studies in Advanced Mathematics Vol. 39, Cam- bridge University Press, Cambridge, 1998. 

\bibitem{CaWa} K. Cameron, T. Walker, {\it The graphs with maximum induced matching and maximum matching the same size}, Discrete Math. {\bf 299} (2005) 49-55.

\bibitem{CV}  A. Constantinescu and M. Varbaro, {\it Koszulness, Krull dimension, and other properties of graph-related algebras},  J. Algebr. Comb., {\bf 34}(2011), 375--400.

\bibitem  {DK} K. Dalili and M. Kummini, {\it Dependence of Betti numbers on characteristic},Comm. Algebra {\bf 42} (2014), no. 2, 563-570.

\bibitem S.A. S.  Fakhari,{\it Stability of depth and Stanley depth of symbolic powers of squarefree monomial ideals}, Proc. Amer. Math. Soc. {\bf 148}(2020), no. 5, 1849-1862.

\bibitem {FHT}  C.A. Francisco, A. Hoefel, and A. Van Tuyl, EdgeIdeals: a package for (hyper)graphs. J. Software Algebra Geom. 1 (2009) 1-4.

\bibitem {GS} D. R. Grayson and M. E. Stillman, Macaulay 2, a software system for research in algebraic geometry. http://www.math.uiuc.edu/Macaulay2/.

\bibitem {HHTT} H.T. H\`{a},  H.D. Nguyen,  N.V. Trung and  T.N. Trung, {\it Depth functions of powers of homogeneous ideals},  Proc. Amer. Math. Soc. {\bf 149} (2021), no. 5, 1837-1844.

\bibitem {HTuyl} H.T. H\`{a} and A. Van Tuyl, {\it Monomial ideals, edge ideals of hypergraphs, and their graded Betti numbers}, J. Algebraic Combin. {\bf 27} (2008), no. 2, 215-245.

\bibitem  {HT} N. T. Hang and T. N. Trung (2017), {\it The behavior of depth functions of cover ideals of unimodular hypergraphs},  Ark. Math., \textbf{55}(1), 8-104.


\bibitem {HHbi} J. Herzog and T. Hibi, {\it The depth of powers of an ideal}, J. Algebra 291 (2005), no. 2, 325-650.

\bibitem {HHT1} J. Herzog, T. Hibi and N. V. Trung, {\it Symbolic powers of monomial ideals and vertex cover algebras},  Adv. Math. {\bf 210}(1) (2007), 304 - 322.

\bibitem {HHT2} J. Herzog, A. A. Qureshi, {\it Persistence and stability properties of powers of ideals}, J. Pure Appl.Algebra 219 (2015), 530–542.

\bibitem {HRV} J. Herzog, A. Rauf and M. Vladoiu, {\it The stable set of associated prime ideals of a polymatroidal ideal}, J. Algebraic Combin. {\bf 37} (2013), no. 2, 289-312.
 
\bibitem {HV} J. Herzog and M. Vladoiu, {\it Squarefree monomial ideals with constant depth function}, J. Pure Appl. Algebra {\bf 217} (2013), no. 9, 1764-1772. 
 
\bibitem {HHKO}  T. Hibi, A. Higashitani, K. Kimura and A. B. O'Keefe , {\it  Algebraic study on Cameron - Walker graphs},  J. Algebra {\bf 422} (2015), 257-269.

\bibitem  {HKTT} L. T. Hoa, K. Kimura, N. Terai and T. N. Trung, {\it Stability of depths of symbolic powers of Stanley-Reisner ideals}, J. Algebra, {\bf 473}(2017), 307-323. 
 
\bibitem  {HT2} L. T. Hoa and T. N. Trung, {\it Partial Castelnuovo-Mumford regularities of sums and intersections of powers of monomial ideals}, Math. Proc. Cambridge Philos Soc., {\bf 149}(2010), 1-18.

\bibitem  {HO}  M. Hochster, {\it Cohen-Macaulay rings, combinatorics, and simplicial complexes, in B. R. Mc- Donald and R. A. Morris (eds.), Ring theory II}, Lect. Notes in Pure and Appl. Math., {\bf 26}(1977), M. Dekker, 171-223.

\bibitem {MS} E. Miller and B. Sturmfels, Combinatorial commutative algebra. Springer, 2005.


%\bibitem {BMVV} J. Mart\'{i}nez-Bernal, S. Morey, R.H. Villarreal, C. E. Vivares (2019), {\it Depth and regularity of monomial
%	ideals via polarization and combinatorial optimization}, Acta Math. Vietnam. {\bf 44}, 243–268.

\bibitem  {MT1} N. C. Minh and N. V. Trung, {\it Cohen-Macaulayness of powers of two-dimensional squarefree monomial ideals}, J. Algebra, {\bf 322}(2009), 4219--4227.

\bibitem  {MT2} N.C. Minh and  N.V. Trung, {\it Cohen-Macaulayness of monomial ideals and symbolic powers of Stanley-Reisner ideals}, Adv. Math. {\bf 226} (2011), no. 2, 1285–1306.


\bibitem {NT} H.D. Nguyen and N.V. Trung, {\it Depth functions of symbolic powers of homogeneous ideals}, Invent.Math., {\bf 218}(3) (2019), 779-827.

\bibitem  {ST} R. P. Stanley, Combinatorics and Commutative Algebra, second edition, Birkhauser, Boston, MA, 1996.
 
\bibitem {T} Y. Takayama, {\it Combinatorial characterizations of generalized Cohen-Macaulay monomial ideals}, Bull. Math. Soc. Sci. Math. Roumanie (N.S.) {\bf 48}(2005), 327--344.

\bibitem  {Tr} T.N. Trung, {\it Regularity, matchings and Cameron - Walker graphs}, Collect. Math. {\bf 71} (2020), no. 1, 83 - 91.

\bibitem  {Tr2} T.N. Trung, {\it Stability of depths of powers of edge ideals}, J. Algebra {\bf 452} (2016), 157-187. 

\bibitem {Z} X. Zheng, {\it Resolutions of facet ideals}, Commun.Algebra {\bf 32} (2004), 2301-2324.
\end{thebibliography}
\end{document}